\documentclass[english,final]{IEEEtran}
\usepackage[T1]{fontenc}
\usepackage{mathrsfs}
\usepackage{bm}
\usepackage{amsmath}
\usepackage{amsthm}
\usepackage{amssymb}
\usepackage{stmaryrd}
\usepackage{graphicx}
 \PassOptionsToPackage{normalem}{ulem}
 \usepackage{ulem}

\makeatletter

\theoremstyle{plain}
\newtheorem{theorem}{\protect\theoremname}
  \theoremstyle{plain}
  
  \theoremstyle{plain}
   \newtheorem{lemma}{\protect\lemmaname}
  \theoremstyle{remark}
  
   \newtheorem{assumption}{\protect\assumptionname}
\theoremstyle{assumption}

\theoremstyle{algorithm}  
  
\usepackage{amsmath}
\usepackage{amssymb}
\usepackage{graphicx,psfrag,cite,subfigure}
\usepackage[table]{xcolor}

\usepackage[acronym]{glossaries}
\newcommand{\newac}{\newacronym}

\makeglossaries

\newac{speb}{SPEB}{square position error bound}
\newac[plural=EFIMs,firstplural=Fisher information matrices (EFIMs)]{efim}{EFIM}{Fisher information matrix}
\newac{ne}{NE}{Nash equilibrium}
\newac{mse}{MSE}{mean squared error}
\newac{toa}{TOA}{time-of-arrival}
\newac{snr}{SNR}{signal-to-noise ratio}
\newac{lan}{LAN}{local area network}
\newac{psd}{PSD}{positive semidefinite}
\newac{pd}{PD}{positive definite}
\newac{wrt}{w.r.t.}{with respect to}
\newac{lhs}{L.H.S.}{left hand side}
\newac{wp1}{w.p.1}{with probability 1}
\newac{kkt}{KKT}{Karush-Kuhn-Tucker}
\newac{wlog}{w.l.o.g.}{without loss of generality}
\newac{mle}{MLE}{maximum likelihood estimation}
\newac{gps}{GPS}{global positioning system}
\newac{rssi}{RSSI}{received signal strength}
\newac{mimo}{MIMO}{multiple-input multiple-output}
\newac{csi}{CSI}{channel state information}
\newac{fdd}{FDD}{frequency division duplexing}
\newac{ms}{MS}{mobile station}
\newac{bs}{BS}{base station}
\newac{d2d}{D2D}{device-to-device}
\newac{slnr}{SLNR}{signal-to-interference-leakage-and-noise-ratio}
\newac{ula}{ULA}{uniform linear antenna array}
\newac{pas}{PAS}{power angular spectrum}
\newac{mmse}{MMSE}{minimum mean square error}
\newac{zf}{ZF}{zero-forcing}
\newac{rzf}{RZF}{regularized zero-forcing}
\newac{as}{AS}{angular spread}
\newac{aod}{AOD}{angle of departure}
\newac{iid}{i.i.d.}{independent and identically distributed} 
\newac{sinr}{SINR}{signal-to-interference-and-noise ratio}
\newac{tdd}{TDD}{time-division duplex}
\newac{rvq}{RVQ}{random vector quantization}
\newac{rhs}{R.H.S.}{right hand side}
\newac{mrc}{MRC}{maximum ratio combining}
\newac{cdf}{CDF}{cumulative distribution function}
\newac{a.s.}{a.s.}{almost surely}
\newac{los}{LOS}{line-of-sight}
\newac{jsdm}{JSDM}{joint spatial division and multiplexing}
\newac{map}{MAP}{maximum a posteriori}
\newac{klt}{KLT}{Karhunen-Lo\`eve Transform}
\newac{lbe}{LBE}{link bargaining equilibrium}
\newac{se}{SE}{Stackelberg equilibrium}
\newac{uav}{UAV}{unmanned aerial vehicle}
\newac{nlos}{NLOS}{non-line-of-sight}
\newac{pdf}{PDF}{probability density function}
\newac{em}{EM}{expectation-maximization}
\newac{knn}{KNN}{$k$-nearest neighbor}
\newac{svd}{SVD}{singular value decomposition}
\newac{nmf}{NMF}{non-negative matrix factorization}
\newac{umf}{UMF}{unimodal-constrained matrix factorization}
\newac{rmse}{RMSE}{rooted mean squared error}

\setkeys{Gin}{width=1.0\columnwidth}

\makeatother

\usepackage{babel}
  \providecommand{\definitionname}{Definition}
  \providecommand{\lemmaname}{Lemma}
  \providecommand{\propositionname}{Proposition}
  \providecommand{\remarkname}{Remark}
\providecommand{\theoremname}{Theorem}
\providecommand{\assumptionname}{Assumption}
\providecommand{\algorithmname}{Algorithm}

\begin{document}

 \title{Dynamic Sensor Subset Selection for Centralized Tracking of a   Stochastic Process
 \thanks{The authors are with the Department of Electrical Engineering, University of 
 Southern California. Email: \{achattop,ubli\}@usc.edu}\\
 \thanks{This work was funded by the following grants: ONR N00014-15-1-2550, 
NSF CNS-1213128, 
NSF CCF-1718560, 
NSF CCF-1410009, 
NSF CPS-1446901, 
AFOSR FA9550-12-1-0215
}
\thanks{Some parts of this paper have previously been accepted in conferences \cite{arpan2017globecom}, \cite{arpan2018isit}.}
}

\author{
Arpan~Chattopadhyay \& Urbashi~Mitra \vspace*{-0.4in}
}

\maketitle
%
%



\ifdefined\SINGLECOLUMN
	\setkeys{Gin}{width=0.5\columnwidth}
	\newcommand{\figfontsize}{\footnotesize} 
\else
	\setkeys{Gin}{width=1.0\columnwidth}
	\newcommand{\figfontsize}{\normalsize} 
\fi

\begin{abstract}
Motivated by the Internet-of-things and sensor networks for cyberphysical systems, the problem of dynamic sensor activation for the centralized tracking of an i.i.d. time-varying process is examined. The tradeoff is between energy efficiency, which decreases with the number of active sensors, and fidelity, which increases with the number of active sensors.  The problem of minimizing the time-averaged mean-squared error over infinite horizon is examined under the constraint of  the mean number of active sensors.  The proposed methods artfully combine  Gibbs sampling and stochastic approximation for learning, in order to create a high performance, energy efficient tracking mechanisms with active sensor selection.  Centralized tracking of  i.i.d. process with known distribution as well as an unknown parametric distribution are considered. For an i.i.d. process with known distribution, convergence to the global optimal solution with high probability is proved. The main challenge of the i.i.d. case is that the process has a distribution parameterized by a known or  unknown parameter which must be learned; one  key theoretical result  proves that the proposed algorithm for tracking an i.i.d. process with unknown parametric distribution converges to local optima. Numerical results show the efficacy of the proposed algorithms and also  suggest that global optimality is in fact achieved in some cases.  
\end{abstract}

\section{Introduction}\label{section:introduction}
Controlling and monitoring physical processes via sensed data are integral parts of internet-of-things (IOT) and cyber-physical systems,   and also have applications  in  industrial process monitoring and control, localization, tracking of mobile objects,  environmental monitoring, system identification and disaster management.  In such applications, sensors are simultaneously resource constrained (power and/or bandwdith) and tasked to achieve high performance sensing, control, communication, and tracking.  Wireless sensor networks must further contend with interference and fading.  One strategy for balancing resource use with performance is to activate a subset of the total possible number of sensors to limit both computation as well as bandwidth use.
  
Herein, we address the fundamental problem of optimal dynamic sensor subset selection for tracking a time-varying stochastic process.   We first examine the centralized tracking of an i.i.d. process with a known distribution, which is a   precursor to the  centralized tracking of an i.i.d. process with an unknown, parametric distribution.   For the known prior distribution case, optimality of the proposed algorithm is proven. For the  proposed algorithm for centralized  tracking of an i.i.d. process with parameter learning, results on almost sure convergence to local optima are proven.    The algorithms are numerically validated to demonstrate their efficacy against competetive algorithms and natural heuristics.

%
Optimal sensor subset selection problems can be broadly classified into two 
categories: (i) optimal sensor subset selection for static data with known
prior distribution, but unknown realization, and (ii)  dynamic sensor subset
selection  to track a time-varying stochastic process. There have been several
recent attempts to solve the first problem; see \cite{wang-etal16efficient-observation-selection} for sensor network applications and \cite{schnitzler-etal15sensor-selection-crowdsensing} for mobile crowdsensing applications. This problem poses two major challenges: (i) computing the estimation
error given the observations from a subset of sensors, and (ii) finding the
optimal sensor subset from exponentially many number of subsets.  In \cite{wang-etal16efficient-observation-selection}, a tractable lower bound on performance addressed the first challenge and a greedy algorithm addressed the second. In our current  paper, we  use Gibbs sampling to solve the problem of tracking an i.i.d. time varying process via active sensing. While estimation of static data and tracking i.i.d. time-varying process are the same problems mathematically, herein we provide a provably optimal alternative approach to that of \cite{wang-etal16efficient-observation-selection}; in case the distribution is unknown and learnt over time, Gibbs sampling also yields a low-complexity sensor subset selection scheme, thereby eliminating the need for running a greedy algorithm whose complexity scales with the number of sensors.

There have   been several related works on the problem of dynamic sensor subset selection  to track a time-varying stochastic process; see \cite{daphney-etal14active-classification-pomdp, daphney-etal13energy-efficient-sensor-selection, krishnamurthy07structured-threshold-policies, wu-arapostathis08optimal-sensor-querying, gupta-etal06stochastic-sensor-selection-algorithm, bertrand-moonen10sensor-selection-linear-mmse}. In  \cite{krishnamurthy07structured-threshold-policies}, the problem of selecting a single sensor node  to track a Markov chain is addressed; \cite{krishnamurthy07structured-threshold-policies} assumes the availability of a centralized controller which has knowledge of the latest observation made by the selected sensor. This problem was extended to sensor subset selection (by a centralized controller)  in \cite{daphney-etal14active-classification-pomdp}, and energy efficiency issues were incorporated in \cite{daphney-etal13energy-efficient-sensor-selection}. These two papers considered sequential  decision making over a finite time horizon. The  existence of an optimal policy for the centralized optimal dynamic sensor subset selection problem for infinite time horizon is proved in \cite{wu-arapostathis08optimal-sensor-querying};   the structure of the optimal policy for the special case of linear quadratic Gaussian (LQG) problem is also provided. The paper \cite{gupta-etal06stochastic-sensor-selection-algorithm} addresses the problem of selecting a single sensor at each time, with the assumption that the observation of the sensor is shared among all   sensors.  Thompson sampling, in  \cite{schnitzler-etal15sensor-selection-crowdsensing},  solved the problem of {\em centralized} tracking of a linear Gaussian process (with unknown noise statistics) via active sensing.

Herein, we   consider the problem of dynamically choosing the optimal sensor subset for centralized   tracking of an i.i.d. time-varying process with an unknown parametric distribution, using tools from   Gibbs sampling (see \cite{breamud99gibbs-sampling}) and stochastic approximation (see \cite{borkar08stochastic-approximation-book}). To the best of our knowledge, this problem has not been  solved in   prior work.    Our work accommodates energy constraint in the network by imposing a constraint on the number of active sensors.

In this paper, we make the following contributions:

\begin{enumerate}
\item    In Section~\ref{section:gibbs-sampling-unconstrained-problem}, a centralized tracking and learning algorithm for an i.i.d. process   with a known  distribution is developed, in order to minimize time-average estimation error subject to a constraint on the mean number of active sensors.  In particular, Gibbs sampling minimizes computational complexity for a relaxed version of the problem, along  with stochastic approximation that is employed to iteratively update a Lagrange multiplier to achieve the mean number of activated sensors constraint.  Desired almost sure convergence to the optimal solution is proved. A challenge we overcome in the analysis,  is handling updates at different time scales that given rise to several technical issues that need to be addressed.

\item    In Section~\ref{section:iid-data}, a centralized tracking and learning algorithm for an i.i.d. process   with an unknown, but parametric distribution is developed. In addition to Gibbs sampling and stochastic approximation as used in Section~\ref{section:gibbs-sampling-unconstrained-problem}, simultaneous perturbation stochastic approximation (SPSA) is employed for parameter estimation obviating the need for   expectation-maximization.

\item  Numerical results show that the proposed algorithms outperform simple greedy algorithms.  Numerical results also demonstrate a tradeoff between performance and computational cost for learning.  Furthermore, the numerical results show that sometimes global optima are achieved in tracking i.i.d. process with unknown parametric distribution.
\end{enumerate}

The rest of the paper is organized as follows. The system model  is described in Section~\ref{section:system-model}. Tracking of an i.i.d. process with known distribution is described in Section~\ref{section:gibbs-sampling-unconstrained-problem}.  Section~\ref{section:iid-data} deals with the tracking of an i.i.d. process with unknown,  parametric distribution.   Numerical results are presented in Section~\ref{section:numerical-results}, followed by the conclusion in Section~\ref{section:conclusion}. All mathematical proofs are provided in the appendices.

\section{System Model}\label{section:system-model}
\begin{figure}[!t]
\begin{center}
\includegraphics[height=5cm, width=7cm]{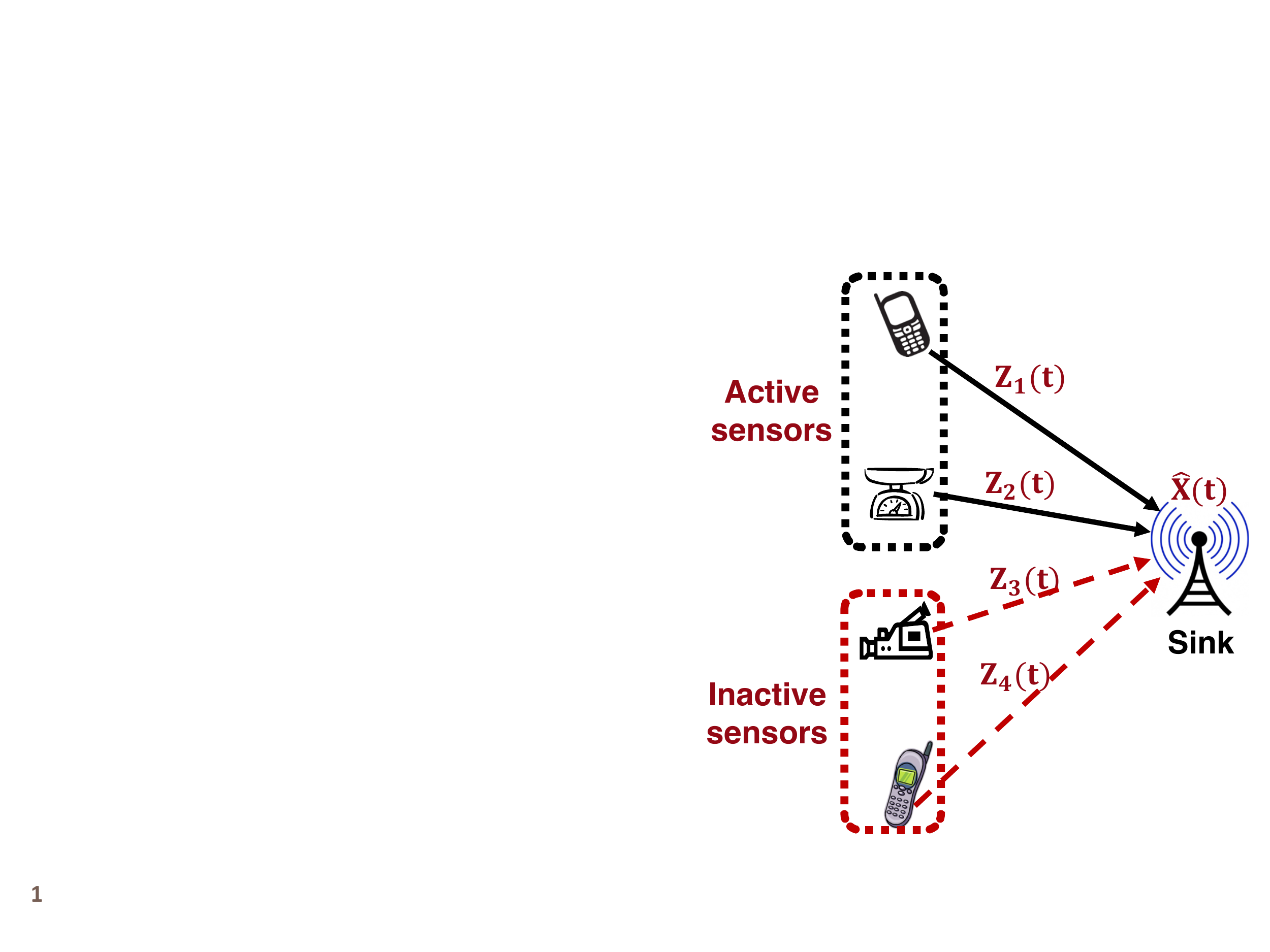}
\end{center}
\caption{Centralized   estimation. Sensors send their observations $\{Z_1(t), \cdots,Z_N(t)\}$ to the   sink, and the fusion center estimates $\hat{X}(t)$. }
\label{fig:network-diagrams}
\vspace{-2mm}
\end{figure}

We consider a  connected single-hop wireless sensor network (see Figure~\ref{fig:network-diagrams}) where sensor nodes communicate directly with the fusion center; the fusion center is responsible for all control or estimation operations in the network.   The sensor nodes are denoted by the index set   $\mathcal{N}=\{1,2,\cdots,N\}$.  While our methods can be adapted to consider multihopped communication via relays, we do not treat this case herein.

The physical process under measurement is denoted by  $\{X(t)\}_{t \geq 0}$, where $t$ is a discrete time index and $X(t) \in \mathbb{R}^{q \times 1}$.    $\{X(t)\}_{t \geq 0}$ is an i.i.d. process. The distribution of $X(t)$ may be known, or $X(t)$ might have a parametric distribution $p_{\theta_0}(\cdot)$, where the {\em unknown} parameter vector $\theta_0$  needs to be be learnt via the  measurements. The parameter vector $\theta_0$ lies inside the interior of a compact subset $\Theta \subset \mathbb{R}^d$.

 At time $t$, if a sensor~$k$ is used to sense the process, then the observation at sensor~$k$ is provided by a $r$-dimensional column vector  
 \begin{eqnarray*}
 z_k(t)& = & f_0 (X(t))+v_k(t),
 \end{eqnarray*}
 where   $v_k(t)$ is a Gaussian random vector (observation noise) which is   independent across $k$ and i.i.d. across $t$. 
 
 Let $B(t) \in \{0,1\}^{1 \times N}:=\mathcal{B}$ be a vector where the $k$-th entry $B_k(t)=1$  if the $k$th sensor is activated at time, $t$ and $B_k(t)=0$, if it is inactive. The decision to activate any sensor for sensing and communicating the observation is taken   by the fusion center. We denote by $\mathcal{B} \doteq\{0,1\}^N$ the set of all possible configurations ({\em i.e.}, sensor activation vectors) in the network, and by $B$ a generic configuration. Clearly, $B(t) \in \mathcal{B}$. Each configuration represents a unique set of activated sensors.  The notation $B_{-j} \in \{0,1\}^{N-1}$ is used to represent the configuration $B$ with its $j$-th entry removed. We denote by $(B_{-j},0)$ another configuration which agrees with $B$ at all coordinates other than the $j$-th coordinate, where $(B_{-j},0)$ has a value $0$ as the $j$-th entry (i.e., the $j$-th sensor is not activated); a similar definition holds for $(B_{-j},1)$.

The observation made by sensor~$k$ at time $t$ is  $Z_k(t)=B_k(t) z_k(t)$. We define $Z(t) \doteq \{Z_k(t): 1 \leq k \leq N\}\}$.

\subsection{Problem framework}
Our sensor network seeks to achieve two goals: develop a sensing strategy, $B(t)$ and compute an estimate of $X(t)$ at the fusion center which is denoted by
$\hat{X}(t)$ (see  Figure~\ref{fig:network-diagrams}). For the case of unknown distribution paramter, the fusion center also computes an estimate of those parameters, $\hat{\theta}(t)$. To compute these three functions we define two distinct information structures:
\begin{eqnarray}
\mathcal{H}_c(t) & = & \{B(\tau), \hat{X}(\tau-1), Z(\tau), \hat{\theta}(\tau-1), \;
\forall \; \tau \leq t\}\\
\mathcal{H}_p(t) & = & \{B(\tau), \hat{X}(\tau), Z(\tau), \hat{\theta}(\tau), \;
\forall \; \tau \leq t-1\}
\end{eqnarray}
The corresponding functions are then given as follows:
\begin{eqnarray}
B(t) & = & \mu_1(\mathcal{H}_p(t))\\
\hat{X}(t) & = & \mu_2(\mathcal{H}_c(t))\\
\end{eqnarray}
We observe the sequential nature in applying the functions $\mu_i$, that is, the activation vector $B(t)$ determines the observations $Z(t)$ which in turn are used to compute the tracked process, $\hat{X}(t)$. For unknown $\theta_0$, we compute $\hat{\theta}(t)= \mu_3(\mathcal{H}_c(t)) $.
For an i.i.d. time varying process, $\mathcal{H}_c(t)$ is sufficient to estimate  $\hat{X}(t)$.  However, in order to optimally decide $B(t)$ when $\theta_0$ is unknown, the fusion center needs  knowledge about the performance of all past configurations. Hence, $\mathcal{H}_p(t)$ and $\mathcal{H}_c(t)$ have two different information structures. However, we will see that, our Gibbs sampling algorithm determines $B(t)$ by using only a sufficient statistic (which captures the past history) calculated iteratively in each slot.

We define a  policy $\mu=(\mu_1,\mu_2,\mu_3 )$ as a tuple of mappings, where $\mu_1(\mathcal{H}_p(t))=B(t)$, $\mu_2(\mathcal{H}_c(t))=\hat{X}(t)$ and $\mu_3(\mathcal{H}_c(t))=\hat{\theta}(t)$ as discussed earlier.  The policy $\mu$ may be randomized, where the quantities $B(t)$, $\hat{\theta}(t)$ and $\hat{X}(t)$ are chosen according to  random distributions defined by $\mu$; that is, $\mu_1(\mathcal{H}_p(t))$, $\mu_2(\mathcal{H}_c(t))$ and $\mu_3(\mathcal{H}_c(t))$ are three probability distributions for $B(t)$, $\hat{X}(t)$ and $\hat{\theta}(t)$, respectively. In the sequel, we will investigate Gibbs sampling strategies for sensor selection; therein, $B(t)$ will be random.

Our   goal is to solve the following centralized   problem of minimizing the time-average mean squared error (MSE) subject to a constraint on the mean number of active sensors per unit time:
\vspace*{-0.2in}
\begin{align}
&& \mu^*=\arg \min_{\mu} \limsup_{t \rightarrow \infty} \frac{1}{t} \sum_{\tau=1}^t \mathbb{E}_{\mu} ||X(\tau)-\hat{X}(\tau)||^2  \nonumber\\
&& s.t.\,\,\,\, \limsup_{t \rightarrow \infty} \frac{1}{t} \sum_{\tau=1}^t  \mathbb{E}_{\mu} ||B(\tau)||_1 \leq \bar{N} \tag{P1} \label{eqn:centralized-constrained-problem}
\end{align}
where $\mathbb{E}_{\mu}$ is the expectation under policy $\mu$, and the expectation is taken over the randomness in the process as well as any possible randomness in the policy $\mu$.

\section{IID proces with known distribution}
\label{section:gibbs-sampling-unconstrained-problem}
In this section, we  provide an algorithm for solving the centralized problem \eqref{eqn:centralized-constrained-problem} when $\{X(t)\}_{t \geq 0}$ is i.i.d. with   known distribution. This algorithm is developed as a precursor to the algorithms for tracking an i.i.d. process  having a parametric distribution  with an unknown parameter $\theta_0$.

\subsection{Relaxing the constrained problem}\label{subsection:relaxed-version-centralized-constrained-problem-iid-data}
In order to solve the constrained problem \eqref{eqn:centralized-constrained-problem}, we 
first relax \eqref{eqn:centralized-constrained-problem} by using a Lagrance multiplier $\lambda$, and obtain the following unconstrained problem:

\begin{equation} \label{eqn:centralized-unconstrained-problem} 
\mu^*=\arg \min_{\mu} \limsup_{t \rightarrow \infty} \frac{1}{t} \sum_{\tau=1}^t \mathbb{E}_{\mu} \bigg(  ||X(\tau)-\hat{X}(\tau)||^2 + \lambda  ||B(\tau)||_1 \bigg) \tag{P2}
\end{equation}
\normalsize
The multiplier $\lambda \geq 0$ can be viewed as the cost incurred for activating a sensor at any time instant. We will see later that solution of the unconstrained problem \eqref{eqn:centralized-unconstrained-problem} will be used to solve the constrained problem \eqref{eqn:centralized-constrained-problem}.

We observe that, at time $\tau$, for the chosen sensor subset $B(\tau)$ and the corresponding collected observations $Z(\tau)$, the minimum mean squared error (MMSE)  estimate of $X(\tau)$ is given by $\hat{X}(\tau)=\mathbb{E}(X(\tau)|\mathcal{H}_c(\tau))$; hence, {\em we fix the estimation policy $\mu_2$ and solve \eqref{eqn:centralized-unconstrained-problem} only over the sensor subset selection policy $\mu_1$ (since the distribution of $X(\tau)$ is known, $\mu_3$ has no relevance here)}. Since \eqref{eqn:centralized-unconstrained-problem} is an unconstrained problem and $X(\tau)$ is i.i.d. across $\tau$, there exists at least one optimizer $B^* \in \mathcal{B}$ (not necessarily unique) for the problem \eqref{eqn:centralized-unconstrained-problem}; if the configuration $B^*$ is chosen at each $t$, the minimum cost of \eqref{eqn:centralized-unconstrained-problem} can be achieved  (follows from the law of large numbers, since the cost incurred over time for a given $\mu_2$  constitutes an i.i.d. sequence whose mean is the optimal cost for \eqref{eqn:centralized-unconstrained-problem}). Hence,   \eqref{eqn:centralized-unconstrained-problem} can be   written as:
\begin{equation}\label{eqn:centralized-unconstrained-problem-with-f-and-h} 
\arg \min_{B \in \mathcal{B}} \underbrace{ \underbrace{\mathbb{E}_{\mu_2,B} ||X(\tau)-\hat{X}(\tau)||^2}_{:=f(B)}+ \lambda ||B||_1 }_{:=h(B)} \tag{P3}
\end{equation}
Here $f(B)$ (for any $B \in \mathcal{B}$) is the MSE under estimation policy $\mu_2$ when the sensor activation vector is $B$; $f(B)$ becomes equal to the MMSE under configuration $B$ if $\mu_2(\mathcal{H}_c(\tau))=\mathbb{E}(X(\tau))| \mathcal{H}_c(\tau)  )$. {\em Our results in this paper will hold for MMSE or any other general estimator.}

The following   result tells us how to choose the optimal $\lambda^*$ to solve   \eqref{eqn:centralized-constrained-problem}.

\begin{theorem}\label{theorem:relation-between-constrained-and-unconstrained-problems}
Consider problem \eqref{eqn:centralized-constrained-problem} and  its relaxed version \eqref{eqn:centralized-unconstrained-problem-with-f-and-h}. If there exists a Lagrange multiplier $\lambda^* \geq 0$ and a $B^* \in \mathcal{B}$, such that an optimal configuration for \eqref{eqn:centralized-unconstrained-problem-with-f-and-h}  under  $\lambda=\lambda^*$ is $B^*$, and the constraint in \eqref{eqn:centralized-constrained-problem}  is satisfied with equality under the pair 
$(B^*,\lambda^*)$, then $B^*$ is an optimal configuration for \eqref{eqn:centralized-constrained-problem}.

In general, if there exist multiple configurations $B_1^*,B_2^*, \cdots, B_m^*$, a multiplier  $\lambda^* \geq 0$, and a probability mass function $(p_1,p_2,\cdots,p_m)$  such that (i) each of $B_1^*,B_2^*,\cdots, B_m^*$ is optimal for problem~\eqref{eqn:centralized-unconstrained-problem-with-f-and-h}  under $\lambda^*$, and (ii) $\sum_{i=1}^m p_i ||B_i^*||_1=\bar{N}$,  then an optimal solution for \eqref{eqn:centralized-constrained-problem} is to select one $B_i^*$ independently according to the probability mass function noted above. 
\end{theorem}

\begin{proof}
See Appendix~\ref{appendix:proof-of-relation-between-constrained-and-unconstrained-problems}.
\end{proof}
Theorem~\ref{theorem:relation-between-constrained-and-unconstrained-problems} allows us to obtain a solution for \eqref{eqn:centralized-constrained-problem} from the solution of \eqref{eqn:centralized-unconstrained-problem-with-f-and-h}   by choosing an appropriate $\lambda^*$; the existence of   $\lambda^*$ will be discussed in  Section~\ref{subsection:gibbs-stochastic-approximation-unconstrained-problem}.

\subsection{Solving \eqref{eqn:centralized-unconstrained-problem} and \eqref{eqn:centralized-unconstrained-problem-with-f-and-h} for known distribution}\label{subsection:finite-beta-gibbs-sampling-unconstrained-problem}
Finding the optimal solution of  \eqref{eqn:centralized-unconstrained-problem} and \eqref{eqn:centralized-unconstrained-problem-with-f-and-h} requires us to search over  $2^N$ possible configurations and to compute the MSE for each   configuration.  Hence, we propose  Gibbs sampling based algorithms to avoid this $O(2^N)$ computation.

Let us define a probability distribution   $\pi_{\beta}(\cdot)$ over $\mathcal{B}$ as (with a parameter $\beta>0$):
\begin{eqnarray}\label{eqn:definition-of-Gibbs-distribution}
\pi_{\beta}(B)\doteq\frac{e^{-\beta h(B)}}{\sum_{B' \in \mathcal{B}}e^{-\beta h(B')}} \doteq \frac{e^{-\beta h(B)}}{Z_{\beta}}.
\end{eqnarray}
Following the terminology in statistical physics,  we call   $\beta$ the {\em inverse temperature}, and $Z_{\beta}$ the {\em partition function}. The quantity $h(B)$ is viewed as the {\em energy under configuration $B$. } 
It is straightforward to see that $\lim_{\beta \uparrow \infty} \sum_{ B \in \arg \min_{A \in \mathcal{B}} h(A) } \pi_{\beta}(B)=1$. Hence, if a configuration $B(t)$ is selected at each time  $t$ with probability distribution $\pi_{\beta}(\cdot)$ for sufficiently large $\beta>0$, then $B(t)$ will belong to the set of minimizers of \eqref{eqn:centralized-unconstrained-problem-with-f-and-h} with high probability; if $\Delta_1:=\min_{B' \in \mathcal{B}:h(B')\neq \min_{B \in \mathcal{B}}h(B)} h(B')-\min_{B \in \mathcal{B}}h(B)$, then, for a {\em unique} minimizer $B^*$, $\pi_{\beta}(B^*) \geq \frac{1}{1+(2^N-1)e^{-\beta \Delta_1}}$ (which goes to $1$ as $\beta \rightarrow \infty$).   However, computing   $Z_{\beta}$ requires $2^N$ addition operations; hence, we use a sequential subset selection algorithm based on Gibbs sampling (see  \cite[Chapter~$7$]{breamud99gibbs-sampling})  in order to avoid explicit computation of $Z_{\beta}$ while picking $B(t) \sim \pi_{\beta}(\cdot)$. 

Below we introduce the Basic Gibbs (BG) algorithm.

\noindent\fbox{
    \parbox{0.47 \textwidth}{
{\bf BG algorithm:}     Start with an initial configuration $B(0)$.  At time~$t$, pick a random sensor~$j_t$ uniformly from the set of all sensors.   Choose $B_{j_t}(t)=1$ with probability 
$p(t):=\frac{ e^{-\beta h(B_{-j_t}(t-1),1)}}{e^{-\beta h(B_{-j_t}(t-1),1)}+e^{-\beta h(B_{-j_t}(t-1),0)}}$ and choose $B_{j_t}(t)=0$ with probability 
$(1-p(t))$. 
For $k  \neq j_t$,  choose $B_k(t)=B_k(t-1)$.  Activate the sensors according to $B(t)$.
    }
}


Note that, in this algorithm,  it is sufficient to maintain $\mathcal{H}_p(t)=B(t-1)$ due to the i.i.d. nature of $X(t)$.

\begin{theorem}\label{theorem:convergence-basic-Gibbs-sampling}
Under  the BG algorithm,  $\{B(t)\}_{t \geq 0}$ is a reversible, ergodic, time-homogeneous Markov chain  with stationary distribution $\pi_{\beta}(\cdot)$.
\end{theorem}

\begin{proof}
Follows from the theory  in \cite[Chapter~$7$]{breamud99gibbs-sampling}). The proof can be done by verifying the detailed balance equation for the Markov chain $B(t)$.
\end{proof}

Theorem~\ref{theorem:convergence-basic-Gibbs-sampling} tells us that if the fusion center runs  BG and reaches the steady state distribution of the Markov chain $\{B(t)\}_{t \geq 0}$, then the configuration chosen by the algorithm will have distribution $\pi_{\beta}(\cdot)$.  Also, by the ergodicity of $\{B(t)\}_{t \geq 0}$, the time-average occurence rates of all configurations match the distribution $\pi_{\beta}(\cdot)$ almost surely.

For very large  $\beta>0$, if one runs $\{B(t)\}_{t \geq 0}$ for a sufficiently long, finite time $T_0$, then the terminal state $B(T_0)$ will belong to 
$\arg \min_{B \in \mathcal{B}} h(B)$ with high probability. 
We have already shown that, for a {\em unique} minimizer $B^*$, $\pi_{\beta}(B^*) \geq \frac{1}{1+(2^N-1)e^{-\beta \Delta_1}}$ (which goes to $1$ as $\beta \rightarrow \infty$). In Section~\ref{subsection:convergence-rate-BG}, we will provide an upper bound on $d_V(\pi^{(t)}, \pi_{\beta})$ which is the total variation distance between the distribution of $B(t)$ under BG algorithm and the distribution $\pi_{\beta}$; the upper bound is $d_V(\pi^{(t)},\pi_{\beta}) \leq d_V(\pi^{(0)},\pi_{\beta}) (1-\frac{e^{-\beta N \Delta}}{N^N})^{\lfloor \frac{t}{N} \rfloor}$, where 
$\Delta:=\max_{B \in \mathcal{B}, A \in \mathcal{B}}|h(B)-h(A)|$.  Hence, for a large but finite time $T_0$, we can derive the following bound: 
\begin{eqnarray*}
\pi^{(T_0)}(B^*) & \geq & \pi_{\beta}(B^*)-2 d_V(\pi^{(T_0)},\pi_{\beta})\\
& \geq & \frac{1}{1+(2^N-1)e^{-\beta \Delta_1}}\\
&&- 2 d_V(\pi^{(0)},\pi_{\beta}) \bigg(1-\frac{e^{-\beta N \Delta}}{N^N} \bigg)^{\lfloor \frac{T_0}{N} \rfloor}
\end{eqnarray*}

\subsection{The exact solution}\label{subsection:growing-beta-gibbs-sampling-unconstrained-problem}
BG  is operated with a fixed $\beta$, but  the optimal solution of the unconstrained problem~\eqref{eqn:centralized-unconstrained-problem} can only be  obtained with 
$\beta \uparrow \infty$; this is done by updating  $\beta$  at a slower time-scale than the iterates of BG algorithm. The quantity $\beta(t)$ is increased logarithmically with time in order to maintain the necessary timescale difference between Gibbs sampling and $\beta(t)$ update. We call this new algorithm Adaptive Basic Gibbs or ABG.

\noindent\fbox{
    \parbox{0.47 \textwidth}{
{\bf ABG algorithm:}  This algorithm is same as BG except that at   time $t$, we use  
$\beta(t):=\beta(0) \log (1+t)$   to compute the update probabilities, where $\beta(0)>0$, $\beta(0) N \Delta<1$, and   $\Delta:=\max_{B \in \mathcal{B}, A \in \mathcal{B}}|h(B)-h(A)|$.
    }
}


\begin{theorem}\label{theorem:result-on-weak-and-strong-ergodicity}
Under the ABG algorithm, the Markov chain $\{B(t)\}_{t \geq 0}$ is strongly ergodic, and the limiting probability distribution satisfies $\lim_{t \rightarrow \infty} \sum_{A \in \arg \min_{C \in \mathcal{B}} h(C) }\mathbb{P}(B(t)=A)=1$.
\end{theorem}

\begin{proof}
See Appendix~\ref{appendix:proof-of-weak-and-strong-ergodicity}.  
We have used the notion of weak and strong ergodicity of time-inhomogeneous Markov chains from 
\cite[Chapter~$6$, Section~$8$]{breamud99gibbs-sampling}), which is  provided in Appendix~\ref{appendix:weak-and-strong-ergodicity}. The proof is similar to the proof of  \cite[Theorem~$2$]{chattopadhyay-etal16gibbsian-caching-arxiv}, but is given here for completeness.
\end{proof}

Theorem~\ref{theorem:result-on-weak-and-strong-ergodicity} shows that we can solve \eqref{eqn:centralized-unconstrained-problem} {\em exactly} if we run ABG for infinite time, in contrast to BG which   provides an approximate solution.

For  i.i.d. time varying $\{ X(t) \}_{t \geq 0}$ with known joint distribution, we can either: (i) find the optimal configuration $B^*$ using ABG off-line  and use $B^*$ for ever, or (ii) run ABG  at the same timescale as $t$, and use the current configuration $B(t)$ for sensor activation; both schemes will minimize the  cost in \eqref{eqn:centralized-unconstrained-problem}. By the strong ergodicity of $\{B(t)\}_{t \geq 0}$, optimal cost will be achieved for \eqref{eqn:centralized-unconstrained-problem} under ABG.

\subsection{Convergence rate of BG and ABG}\label{subsection:convergence-rate-BG}
Let $\pi^{(t)}$ denote the probability distribution of $B(t)$ under BG.   
Let us consider the transition probability matrix $P$ of the Markov chain $\{Y(l)\}_{l \geq 0}$ with  $Y(l)=B(lN)$, under BG. Let us recall the definition of the Dobrushin's ergodic coefficient $\delta(P)$ from \cite[Chapter~$6$, Section~$7$]{breamud99gibbs-sampling} for the matrix $P$; using a method similar to that of the proof of Theorem~\ref{theorem:result-on-weak-and-strong-ergodicity}, we can show that $\delta(P) \leq ( 1-\frac{ e^{-\beta N \Delta} }{N^N})$. 
  Then, by \cite[Chapter~$6$, Theorem~$7.2$]{breamud99gibbs-sampling}, we can say that under  BG, we have $d_V(\pi^{(lN)},\pi_{\beta}) \leq d_V(\pi^{(0)},\pi_{\beta}) \bigg( 1-\frac{ e^{-\beta N \Delta} }{N^N}  \bigg)^l$. We can prove similar bounds for any $t=lN+k$, where $0 \leq k \leq N-1$.
  
Clearly, under the BG algorithm, the convergence rate decreases as $\beta$ increases. Hence, there is a trade-off between convergence rate and accuracy of the solution in this case. Also, the rate of convergence decreases with $N$. 

Such a closed-form convergence rate bound for ABG is not easily available. 
For the ABG algorithm,  the convergence rate is expected to decrease with time, since the value of $\beta(t)$ increases to $\infty$.

\subsection{Gibbs sampling and stochastic approximation for \eqref{eqn:centralized-constrained-problem}}
\label{subsection:gibbs-stochastic-approximation-unconstrained-problem}

In Section~\ref{subsection:finite-beta-gibbs-sampling-unconstrained-problem} and Section~\ref{subsection:growing-beta-gibbs-sampling-unconstrained-problem}, we presented Gibbs sampling based algorithms for the unconstrained problem  \eqref{eqn:centralized-unconstrained-problem}. Now we provide an algorithm that updates $\lambda$ with time in order to meet the constraint in 
\eqref{eqn:centralized-constrained-problem} with equality, and thereby solves \eqref{eqn:centralized-constrained-problem}  (see Theorem~\ref{theorem:relation-between-constrained-and-unconstrained-problems}) by solving the unconstrained problem. 

Let us denote the optimal configuration for \eqref{eqn:centralized-unconstrained-problem-with-f-and-h} under a given estimation strategy $\mu_2$ (which could be the MMSE estimator) by $B^*$.

\begin{lemma}\label{lemma:active-sensors-decreasing-in-lambda}
For the unconstrained problem~\eqref{eqn:centralized-unconstrained-problem-with-f-and-h}, the optimal mean number of active sensors, $\mathbb{E}_{\mu_2}||B^*||_1$, decreases with $\lambda$. Similarly, the optimal 
error, $\mathbb{E}_{\mu_2}f(B^*)$, increases with $\lambda$.
\end{lemma}
\begin{proof}
See Appendix~\ref{appendix:proof-of-active-sensors-decreasing-in-lambda}.
\end{proof}

Lemma~\ref{lemma:active-sensors-decreasing-in-lambda} provides   intuition as to how to update $\lambda$ in  BG or in ABG in order  to solve \eqref{eqn:centralized-constrained-problem}.   We  seek to provide one  algorithm which updates $\lambda(t)$ at each time instant, based on the number of active sensors in the previous time instant. In order to maintain the necessary timescale difference between the $\{B(t)\}_{t \geq 0}$ process and the $\lambda(t)$ update process, we use stochastic approximation (\cite{borkar08stochastic-approximation-book}) based update rules for $\lambda(t)$. 

The optimal mean number of active sensors, $\mathbb{E}_{\mu_2}||B^*||_1$, for the unconstrained problem~\eqref{eqn:centralized-unconstrained-problem-with-f-and-h} is a decreasing staircase function of $\lambda$, where each point of discontinuity is associated with a change in the optimizer $B^*(\lambda)$. 
Hence, the optimal solution of the constrained problem~\eqref{eqn:centralized-constrained-problem} requires us to randomize between two values of $\lambda$ (and therefore between two configurations) in case the optimal $\lambda^*$ as in 
Theorem~\ref{theorem:relation-between-constrained-and-unconstrained-problems} belongs to the set of such discontinuities. However, this randomization will require us to update a randomization probability at another timescale; having stochastic approximations running in multiple timescales leads to  slow convergence. Hence, instead of using a varying $\beta(t)$, we use a fixed, but large $\beta$ and update $\lambda(t)$ in an iterative fashion using stochastic approximation; BG itself is a  randomized subset selection algorithm and hence no further randomization is required to meet the constraint in \eqref{eqn:centralized-constrained-problem} with equality.  This observation is formalized in the following lemma. This lemma will be crucial in the convergence proof of  the Gibbs Learn (GL) algorithm proposed later.

\begin{lemma}\label{lemma:active-sensors-decreasing-in-lambda-under-basic-gibbs-sampling}
Under BG,  
$\mathbb{E}_{\mu_2} ||B(t)||_1$ is a Lipschitz continuous and decreasing function of $\lambda$.
\end{lemma}

\begin{proof}
See Appendix~\ref{appendix:proof-of-active-sensors-decreasing-in-lambda-under-basic-gibbs-sampling}.
\end{proof}

We make the following feasibility assumption for \eqref{eqn:centralized-constrained-problem}, under BG with the chosen $\beta>0$. 
\begin{assumption}\label{assumption:existence-of-optimal-lambda}
There exists $\lambda^* \geq 0$ such that the constraint in \eqref{eqn:centralized-constrained-problem} under 
$\lambda^*$ and BG is met with equality.
\end{assumption}
Note that, by Lemma~\ref{lemma:active-sensors-decreasing-in-lambda-under-basic-gibbs-sampling}, $\mathbb{E}_{\mu_2}||B||_1$ continuously decreases in $\lambda$. Hence, if $\bar{N}$ is feasible, then such a $\lambda^*$ must exist by the {\em intermediate value theorem}. Our proposed   Gibbs Learn (GL) algorithm updates $\lambda(t)$ iteratively in order to  solve \eqref{eqn:centralized-constrained-problem}. Let us define:
$h_{\lambda(t)}(B):=f(B)+\lambda(t) ||B||_1$ (recall the notation from \eqref{eqn:centralized-unconstrained-problem-with-f-and-h}). Now, we formally describe the GL algorithm to solve the constrained problem \eqref{eqn:centralized-constrained-problem}.

\noindent\fbox{
    \parbox{0.47 \textwidth}{
{\bf GL algorithm:}  \begin{enumerate}

\item  Choose any initial $B(0) \in \{0,1\}^N$ and $\lambda(0) \geq 0$. 
 
\item At each discrete time instant $t=0,1,2,\cdots$, pick a random sensor $j_t \in \mathcal{N}$ independently and uniformly. For sensor $j_t$, choose $B_{j_t}(t)=1$ with probability 
$p:=\frac{ e^{-\beta h_{\lambda(t)}(B_{-j_t}(t-1),1)}}{e^{-\beta h_{\lambda(t)}(B_{-j_t}(t-1),1)}+e^{-\beta h_{\lambda(t)}(B_{-j_t}(t-1),0)}}$ and choose $B_{j_t}(t)=0$ with probability 
$(1-p)$. For $k  \neq j_t$, we choose $B_k(t)=B_k(t-1)$. 

\item Update  $\lambda(t)$ at each node as follows: 

$$\lambda(t+1)=[\lambda(t)+a(t) (||B(t-1)||_1-\bar{N})]_b^c$$

The stepsize $\{a(t)\}_{t \geq 1}$ constitutes a positive sequence such that $\sum_{t=1}^{\infty}a(t)=\infty$  and $\sum_{t=1}^{\infty}a^2(t)<\infty$. The nonnegative projection boundaries $b$ and $c$ are such that  $\lambda^* \in (b,c)$ where $\lambda^*$ is defined in Assumption~\ref{assumption:existence-of-optimal-lambda}. 
\end{enumerate}
    }
}

%
%
%
%
%

{\em We next make two observations on the GL algorithm:}
\begin{itemize}
\item If $||B(t-1)||_1$ is more than $\bar{N}$, then $\lambda(t)$ is increased with the hope that this will reduce the number of active sensors in subsequent time slots, as suggested by Lemma~\ref{lemma:active-sensors-decreasing-in-lambda-under-basic-gibbs-sampling}.
\item The $B(t)$ and $\lambda(t)$ processes run on two different timescales; $B(t)$ runs in the faster timescale whereas $\lambda(t)$ runs in a slower timescale. This can be understood from the fact that the stepsize in the $\lambda(t)$ update process decreases with time $t$. Here the faster timescale iterate will view the slower timescale iterate as quasi-static, while the slower timescale iterate will view the faster timescale as almost equilibriated. This is reminiscent of  two-timescale stochastic approximation (see \cite[Chapter~$6$]{borkar08stochastic-approximation-book}).
\end{itemize}

 Let $\pi_{\beta| \lambda^*}(\cdot)$ denote  $\pi_{\beta}(\cdot)$ under $\lambda=\lambda^*$.

\begin{theorem}\label{theorem:optimality-of-the-learning-algorithm-for-constrained-problem}
Under GL  and Assumption~\ref{assumption:existence-of-optimal-lambda}, we have $\lambda(t) \rightarrow \lambda^*$ almost surely, and the limiting distribution of $\{B(t)\}_{t \geq 0}$ is $\pi_{\beta| \lambda^*}(\cdot)$.
\end{theorem}

\begin{proof}
See Appendix~\ref{appendix:proof-of-optimality-of-the-learning-algorithm-for-constrained-problem}.
\end{proof}

Theorem~\ref{theorem:optimality-of-the-learning-algorithm-for-constrained-problem} says that GL   produces a configuration from the distribution $\pi_{\beta| \lambda^*}(\cdot)$ under steady state. Hence, GL meets the sensor activation constraint in \eqref{eqn:centralized-constrained-problem} with equality and offers a near-optimal time-average mean squared error for the constrained problem; the gap from the optimal  
MSE can be made arbitrarily small by choosing $\beta$ large enough.

\subsection{A hard constraint on the number of activated sensors}\label{subsection:hard-constraint}
Let us consider the following modified constrained problem (recall notation from \eqref{eqn:centralized-unconstrained-problem-with-f-and-h}):
\begin{align}\label{eqn:constrained-optimization-problem-static-data-parametric-distribution}
\min_{B \in \mathcal{B}} f(B) \textbf{ s.t. } ||B||_1 \leq \bar{N} 
\tag{P4}
\end{align}
It is easy to see that \eqref{eqn:constrained-optimization-problem-static-data-parametric-distribution} can be easily solved using similar Gibbs sampling algorithms as in Section~\ref{section:gibbs-sampling-unconstrained-problem}, where the Gibbs sampling algorithm runs only on the set of configurations which activate $\bar{N}$ number of sensors. 
Thus, as a by-product, we have also proposed a methodology for the problem in \cite{wang-etal16efficient-observation-selection}, though our framework is more general than   \cite{wang-etal16efficient-observation-selection}. 

Note that, the constraint in \eqref{eqn:centralized-constrained-problem} is weaker than  \eqref{eqn:constrained-optimization-problem-static-data-parametric-distribution}. Also, if  we choose $\beta$ very large, then the number of sensors activated by GL will have very small variance. This allows us to meet the constraint in \eqref{eqn:constrained-optimization-problem-static-data-parametric-distribution} with high probability.

\section{IID   process with parametric distribution: Unknown $\theta_0$}
\label{section:iid-data}
In Section~\ref{section:gibbs-sampling-unconstrained-problem}, we described algorithms for centralized tracking of an i.i.d. process $\{X(t)\}_{t \geq 0}$ with known distributions. In this section, we will deal with the centralized tracking of an i.i.d. process $\{X(t)\}_{t \geq 0}$ where $X(t) \sim p_{\theta_0}(\cdot)$ with an unknown parameter  $\theta_0 \in \Theta$; in this case, $\theta_0$ has to be learnt over time through observations.

The algorithm described in this section will be an adaptation of the GL algorithm discussed in Section~\ref{subsection:gibbs-stochastic-approximation-unconstrained-problem}. However, when $\theta_0$ is unknown, we have to update  its estimate $\theta(t)$  over time using the sensor observations. In order to solve the constrained problem \eqref{eqn:centralized-constrained-problem}, we still need to update $\lambda(t)$ over time so as to attain the optimal $\lambda^*$ of Theorem~\ref{theorem:relation-between-constrained-and-unconstrained-problems} iteratively. However, $f(B)$ (MSE under configuration $B$) in \eqref{eqn:centralized-unconstrained-problem-with-f-and-h} is unknown since $\theta_0$ is unknown, and its estimate $f^{(t)}(B)$ has to be learnt over time using the sensor observations; as a result $h^{(t)}(B):=f^{(t)}(B)+\lambda(t) ||B||_1$ is also iteratively updated for all $B \in \mathcal{B}$.  Hence, we combine the Gibbs sampling algorithm with update schemes for $f^{(t)}(B)$, $\lambda(t)$ and $\theta(t)$  using multi-timescale  stochastic approximation (see \cite{borkar08stochastic-approximation-book}). The Gibbs sampling runs in the fastest timescale and the $\theta(t)$ update runs in the slowest timescale. 

Since the algorithm has several steps (such as Gibbs sampling, $\theta(t)$ update, $f^{(t)}(B)$ update and $\lambda(t)$ update) and each of these steps needs detailed explanation,  we first describe in details some key features and steps of the algorithm in Section~\ref{subsection:step-size-iid-centralized}, Section~\ref{subsection:gibbs-sampling-step-iid-centralized}, Section~\ref{subsection:theta-update}, Section~\ref{subsection:lambda-update-in-GPL} and Section~\ref{subsection:MSE-update-in-GPL}, and then provide a brief summary of the   algorithm in Section~\ref{subsection:GPL-algorithm}.

 The proposed Gibbs Parameter Learning (GPL) algorithm also requires a sufficiently large positive number $A_0$ and a large integer $T$ as input. The need of these parameters  will be clear as we proceed through the algorithm description. 

Let $\mathcal{J}(t)$ denote the indicator that time $t$ is an integer multiple of $T$.  Define $\nu(t):=\sum_{\tau=0}^t \mathcal{J}(\tau)$ to be the number of time slots till time $t$ when all sensors are activated. 
\subsection{Step size sequences}\label{subsection:step-size-iid-centralized}
 For the stochastic approximation updates of $f^{(t)}(B)$, $\lambda(t)$ and $\theta(t)$, the algorithm uses four nonnegative sequences   $\{a(t)\}_{t \geq 0}$, $\{b(t)\}_{t \geq 0}$, $\{c(t)\}_{t \geq 0}$, $\{d(t)\}_{t \geq 0}$. Let $\{s(t)\}_{t \geq 0}$ be a generic sequence where $s \in \{a,b,c\}$. Then we require the following two conditions:
(i) $\sum_{t=0}^{\infty}s(t)= \infty ,$ and 
(ii) $\sum_{t=0}^{\infty}s^2(t)<\infty,   $.

In addition, we have the following specific assumptions:\\
(iii) $\lim_{t \rightarrow \infty} d(t)=0$, \\
(iv) $\sum_{t=0}^{\infty}\frac{c^2(t)}{d^2(t)}<\infty$,\\
(v) $\lim_{t \rightarrow \infty}\frac{b(t)}{a(t)}=\lim_{t \rightarrow \infty}\frac{c(\lfloor \frac{t}{T} \rfloor )}{b(t)}=0$.

Let us recall from the GL algorithm that we had one stochastic approximation update for  $\lambda(t)$ and a single stepsize sequence $\{a(t)\}_{t \geq 0}$.  However, in the GPL algorithm, we will have three stochastic approximation updates and hence there are three step size sequences $a(t)$, $b(t)$ and $c(t)$ in the sequel; the stepsize $d(t)$ is used in estimate update $\theta(t)$. 

Conditions (i) and (ii) are standard requirements for stochastic approximation step sizes. Conditions (iii) and (iv) are additional requirements for the $\theta(t)$ update scheme  described later. Condition~(v)  ensures that the three update equations run on separate timescales. The stepsize  $c(t)$ corresponds to the slowest timescale and $a(t)$ corresponds to the fastest timescale among these step size sequences; this can be understood from the fact that any iteration involving $c(t)$ as the stepsize will vary very slowly, and the iteration involving $a(t)$  will vary fast due to the large step sizes. In condition~(v), we have $c(\lfloor \frac{t}{T} \rfloor )$ instead of $c(t)$ because $\theta(t)$ in our proposed GPL algorithm will be updated only once in every $T$ time slots, and hence a step size $c(\lfloor \frac{t}{T} \rfloor )$  will be used to update $\theta(t)$ by using observations from all sensors whenever $\mathcal{J}(t)=1$.

\subsection{Gibbs sampling step for sensor subset selection}\label{subsection:gibbs-sampling-step-iid-centralized}
The algorithm also maintains a running estimate $h^{(t)}(B)$ of $h(B)$ for all $B \in \mathcal{B}$. 
At time~$t$, it selects a random sensor $j_t \in \mathcal{N}$   uniformly and independently, and sets  $B_{j_t}(t)=1$ with probability 
$p(t):=\frac{ e^{-\beta h^{(t)}(B_{-j_t}(t-1),1)}}{e^{-\beta h^{(t)}(B_{-j_t}(t-1),1)}+e^{-\beta h^{(t)}(B_{-j_t}(t-1),0)}}$ and  $B_{j_t}(t)=0$ with probability 
$(1-p(t))$. For $k  \neq j_t$, it sets $B_k(t)=B_k(t-1)$. \footnote{This randomization operation can even be repeated multiple times in each time slots to achieve faster convergence results.} The sensors are activated according to   $B(t)$, and the observations  $Z_{B(t)}(t):=\{z_k(t):B_k(t)=1\}$ are collected. Then the algorithm declares $\hat{X}(t)=\mu_2(\mathcal{H}_c(t))$.

\subsection{Parameter estimate update $\theta(t)$}\label{subsection:theta-update}

If $\mathcal{J}(t)=1$, the fusion center reads all sensors and obtains $Z(t)$. This is required primarily because we seek to update $\theta(t)$ iteratively and reach a local maximum of the function:
$$g(\theta)= \mathbb{E}_{X(t) \sim p_{\theta_0}(\cdot), B(t)=[1,1,\cdots,1]}\log p (Z(t) | \theta)$$
$g(\theta)$ is the expected log-likelihood of $Z(t)$ given $\theta$, when all sensors are active and $X(t) \sim p_{\theta_0}(\cdot)$. Note that, maximizing $g(\theta)$ minimizes the KL divergence $D(p(Z(t)|\theta_0)||p(Z(t)|\theta))$.   If we use only the sensor observations corresponding to the activation vector $B(t)$ obtained from Gibbs sampling, then the estimate will be biased by the reading of the sensors which are chosen more frequently by  Gibbs sampling.  However, we will later see that the additional amount of sensing and communication can be made arbitrarily small by choosing $T$ large enough.

Since we seek to reach a local maximum of $g(\theta)= \mathbb{E}_{X(t) \sim p_{\theta_0}(\cdot), B(t)=[1,1,\cdots,1]}\log p (Z(t) | \theta)$, a gradient ascent scheme is employed. The gradient of $g(\theta)$ along any coordinate can be computed by perturbing $\theta$ in two opposite directions along that coordinate and evaluating the difference of  $g(\cdot)$ at those two perturbed values.  However, if $\theta_0$ is high-dimensional, then estimating this gradient along all coordinates is computationally intensive. Moreover, evaluating $g(\theta)$ for any $\theta$ requires us to compute an expectation over the distribution $p_{\theta_0}(\cdot)$ and the distribution of the observation noise at all sensors, which might also be expensive. Hence, we perform a noisy gradient estimation for   $g(\theta)$ by simultaneous perturbation stochastic approximation (SPSA) as in \cite{spall92original-SPSA}. Our algorithm generates 
$\Delta(t) \in \{1,-1\}^d$ uniformly over all sequences, and perturbs  the current estimate $\theta(t)$ by a random vector $d(\nu(t)) \Delta(t)$  (recall that $\nu(t):=\sum_{\tau=0}^t \mathcal{J}(\tau)$) in two opposite directions to obtain $\theta(t)+d(\nu(t))\Delta(t)$ and $\theta(t)-d(\nu(t))\Delta(t)$, and estimates each component of the gradient from the following difference:
$$\log p(Z(t)|  \theta(t)+d(\nu(t)) \Delta(t)  ) -\log p(Z(t)|  \theta(t)-d(\nu(t)) \Delta(t)  ) $$ 
This estimate is noisy because (i) $Z(t)$ and $\Delta(t)$ are random, and (ii) $d(\nu(t))>0$ while ideally it should be infinitesimally small.   

The $k$-th component of $\theta(t)$ is updated as follows:

\begin{eqnarray}\label{eqn:theta-update-iid-data}
&& \theta_k(t+1) \nonumber\\
&=& \bigg[ \theta_k(t)+c(\nu(t)) \mathcal{J}(t)  \bigg( \frac{  \log p(Z(t)|  \theta(t)+d(\nu(t)) \Delta(t)  )   }{2 d(\nu(t)) \Delta_k(t)} \nonumber\\
&& - \frac{  \log p(Z(t)|   \theta(t)-d(\nu(t)) \Delta(t)  )   }{2 d(\nu(t)) \Delta_k(t)} \bigg)  \bigg]_{\Theta}
\end{eqnarray}
The iterates are projected onto the compact set $\Theta$ to ensure boundedness. 

Note that, \eqref{eqn:theta-update-iid-data} is a stochastic gradient ascent iteration  performed once in every $T$ slots with step size $c(\nu(t))$. On the other hand, the sequence $\{d(t)\}_{t \geq 0}$ is the perturbation sequence used in gradient estimation, and hence it need not behave like a standard stochastic approximation step size sequence. However, $d(t)$ should converge to $0$ as $t \rightarrow \infty$, in order to ensure that the gradient estimate is asymptotically unbiased. The conditions (iii) and (iv) in Section~\ref{subsection:step-size-iid-centralized} are technical conditions required for the   convergence of \eqref{eqn:theta-update-iid-data}.

\subsection{Weight update $\lambda(t)$}\label{subsection:lambda-update-in-GPL}
$\lambda(t)$ is updated as follows:
\begin{eqnarray}\label{eqn:lambda-update-iid-data}
\lambda(t+1)=[\lambda(t)+b(t) (||B(t)||_1-\bar{N})]_0^{A_0}.
\end{eqnarray}
The intuition here is that, if $||B(t)||_1>\bar{N}$, the sensor activation cost $\lambda(t)$ needs to be increased to prohibit activating large number of sensors in future; this is motivated by Lemma~\ref{lemma:active-sensors-decreasing-in-lambda-under-basic-gibbs-sampling} and is similar to the $\lambda(t)$ update equation in GL algorithm. The goal is to converge to $\lambda^*$ as defined in Theorem~\ref{theorem:relation-between-constrained-and-unconstrained-problems}.

\subsection{MSE estimate update $f^{(t)}(B)$}\label{subsection:MSE-update-in-GPL}
Since $p_{\theta_0}(\cdot)$ is not known initially, the true value of $f(B)$ (i.e., the MSE under configuration $B$ and  known distribution of $X(t)$)   is not known;  hence, the proposed GPL algorithm updates an estimate $f^{(t)}(B)$ using the sensor observations. 
If $\mathcal{J}(t)=1$, the fusion center obtains $Z(t)$ by reading all sensors. The goal is to obtain a  random sample  $Y_B(t)$ of the MSE under a configuration $B$, by using these observations, and update $f^{(t)}(B)$ using $Y_B(t)$. 

However, since $\theta_0$ is unknown and only $\theta(t)$ is available, as an alternative to the MSE under configuration $B$, the fusion center uses the trace of the conditional covariance matrix of $X(t)$ given $Z_B(t)$,  assuming that $X(t) \sim p(\cdot|\theta(t),Z_B(t))$. Hence, we define a random variable:   
\begin{equation}\label{eqn:Y_B-definition}
Y_B(t):= \mathbb{E}_{X(t) }(||X(t)-\hat{X}_B(t)||^2|Z_B(t), \theta(t))
\end{equation}
for each $B \in \mathcal{B}$, 
where    $\hat{X}_B(t)$ is the MMSE estimate declared by $\mu_2$ for configuration $B$ and given     the observation $Z_B(t)$ made by active sensors determined by $B$, under the assumption that $X(t) \sim p_{\theta(t)}(\cdot)$. Clearly, $Y_B(t)$ is a random variable with the randomness coming from two sources: (i) randomness of $\theta(t)$, and (ii) randomness of $Z_B(t)$ which has a distribution $p(Z_B(t)|\theta_0)$ since the original $X(t)$ process that yields $Z_B(t)$ has a distribution $p_{\theta_0}(\cdot)$. Computation of $Y_B(t)$ is simple for Gaussian $X(t)$ and the MMSE estimator, since closed form expressions are available for  $Y_B(t)$. In case the  computation of $Y_B(t)$ is   expensive (since it requires us to evaluate an expectation over the conditional distribution of $X(t)$ given $Z_B(t)$ and $\theta(t)$), one can obtain an unbiased estimate of $Y_B(t)$ by drawing a random sample of $X(t)$ from the distribution $p_{\theta(t)}(\cdot)$, and scaling the sample squared error $||X(t)-\hat{X}_B(t)||^2$ by the ratio of the two distributions $p(\cdot|\theta(t),Z_B(t))$  and $p_{\theta(t)}(\cdot)$; this technique is basically  importance sampling (see \cite{srinivasan2013importance}) with only one sample.

Using $Y_B(t)$, the following   update is made for all $B \in \mathcal{B}$:
\begin{eqnarray} \label{eqn:fB-update-iid-data}
f^{(t+1)}(B)= [f^{(t)}(B)+\mathcal{J}(t)  a(\nu(t))   (Y_B(t) - f^{(t)}(B)) ]_0^{A_0}
\end{eqnarray}
The iterates are projected onto  a compact interval $[0,A_0]$ to ensure boundedness. The goal here is that, if  $\theta(t) \rightarrow \theta^*$, then $f^{(t)}(B)$ will converge to $\mathbb{E}_{Z_B(t) \sim p(\cdot| \theta_0)} \mathbb{E}_{X(t) \sim p(\cdot|\theta^*,Z_B(t))}(||X(t)-\hat{X}_B(t)||^2|Z_B(t), \theta^*) $, which is equal to  $f(B)$ under $\theta^*=\theta_0$.

{\em We will later argue that this occasional $O(2^N)$ computation for all $B \in \mathcal{B}$ can be avoided, but   convergence   will be slow.}

\subsection{The GPL algorithm}\label{subsection:GPL-algorithm} A summary of all the steps of the GPL algorithm   is provided   below.   {\em We will show in Theorem~\ref{theorem:convergence-of-GLEM} the almost sure convergence of this algorithm.

\noindent\fbox{
    \parbox{0.47 \textwidth}{
{\bf GPL algorithm:}     
 Initialize all iterates arbitrarily.\\
  For any  time  $t=0,1,2,\cdots$, do the following:\\
\begin{enumerate}
\item Sensor activation and process estimation: Perform the Gibbs sampling step as in Section~\ref{subsection:gibbs-sampling-step-iid-centralized} and obtain the activation vector $B(t)$. Activate the sensors $\{k \in {1,2,\cdots,N}: B_k(t)=1\}$ and  obtain the corresponding observations $Z_{B(t)}(t)$ at the fusion center. Estimate $\hat{X}(t)$ using the observations $Z_{B(t)}(t)$ and the current parameter estimate $\theta(t)$. If $\mathcal{J}(t)=1$, read  all sensors and obtain $Z(t)$. Compute or estimate $Y_B(t)$ (defined in \eqref{eqn:Y_B-definition})   for all $B \in \mathcal{B}$.

\item  $f^{(t)}(B)$ update: If $\mathcal{J}(t)=1$, update $f^{(t)}(B)$ for all $B \in \mathcal{B}$ using \eqref{eqn:fB-update-iid-data} and also update $h^{(t)}(B)=f^{(t)}(B)+\lambda(t) ||B||_1$ for all $B \in \mathcal{B}$. 

\item  $\theta(t)$ update: If $\mathcal{J}(t)=1$,  update $\theta(t)$ using \eqref{eqn:theta-update-iid-data}. If $\mathcal{J}(t)=0$, then $\theta(t+1)=\theta(t)$.

\item   $\lambda(t)$ update: Update $\lambda(t)$ according to \eqref{eqn:lambda-update-iid-data}.
\end{enumerate}
    }
}

%
%
%
%
%

Note that, in the GPL algorithm, it is sufficient to consider  $\mathcal{H}_p(t)=\{\lambda(t);B(t-1); f^{(t)}(B) \forall B \in \mathcal{B} \}$, and $\mathcal{H}_c(t)=\{B(t); Z(t); \theta(t)\}$.  The Gibbs sampling step tries to minimize a running estimate $h^{(t)}(\cdot)$ of the unconstrained cost function over the space of configurations $\mathcal{B}$.

  Multiple timescales: {\em GPL has multiple iterations running in multiple timescales (see \cite[Chapter~$6$]{borkar08stochastic-approximation-book}). The $\{B(t)\}_{t \geq 0}$ process runs ar the fastest timescale, whereas the $\{ \theta(t)\}_{t \geq 0}$ update scheme runs at the slowest timescale. The basic idea is that a faster timescale iterate views a slower timescale iterate as quasi-static, whereas a slower timescale iterate views a faster timescale iterate as almost equilibriated. For example, since $\lim_{t \rightarrow \infty}\frac{c(t)}{a(t)}=0$, the $\theta(t)$ iterates will vary very slowly compared to $f^{(t)}(B)$ iterates; as a result, $f^{(t)}(B)$ iterates will view quasi-static $\theta(t)$. In other words, the iterates will behave as if a slower timescale iterate varies in a slow outer loop, and a faster timescale iterate varies in an inner loop.}

\subsection{Complexity of GPL and reducing the complexity}
{\underline{Sampling and communication complexity:} }
{\em Since all   sensors are activated when $\mathcal{J}(t)=1$, the mean number of additional active sensors   per unit time is $O(\frac{N}{T})$; these additional $O(\frac{N}{T})$ observations  need to be communicated to the fusion center. However, the additional $O(\frac{N}{T})$ sensing  can be made large enough by choosing a very large $T$.}

\underline{Computational complexity:}  
{\em The computation of $Y_B(t)$ in \eqref{eqn:fB-update-iid-data} for all $B \in \mathcal{B}$ requires $O(2^N)$ expectation computations whenever $\mathcal{J}(t)=1$. However, if one chooses large  $T$ (e.g., $O(4^N)$), then this additional  computation per unit time will be small. However, if one wants to avoid that computation also, then, when $\mathcal{J}(t)=1$, one can simply compute $Y_{B(t)}(t)$ and update $f^{(t)}(B(t))$ instead of doing it for all configurations $B \in \mathcal{B}$. However,  the stepsize sequence $a(\nu(t))$ cannot be used; instead, a stepsize $a(\nu_B(t))$ has to be used when $B(t)=B$ and $f^{(t)}(B)$ is updated using \eqref{eqn:fB-update-iid-data}, where $\nu_B(t):=\sum_{\tau=0}^t \mathcal{J}(\tau) \mathbb{I}(B(\tau)=B)$.  In this case, the convergence result (Theorem~\ref{theorem:convergence-of-GLEM}) on GPL will still hold; however, the proof will require a technical condition $\lim \inf_{t \rightarrow \infty} \frac{\nu_B(t)}{t}>0$ almost surely for all $B \in \mathcal{B}$, which will be satisfied by the Gibbs sampler using finite $\beta$ and bounded $h^{(t)}(B)$. However, we discuss only \eqref{eqn:fB-update-iid-data} update in this paper for the sake of simplicity in the convergence proof, since technical details of asynchrounous stochastic approximation required in the variant mentioned in this subsection are not the main theme of this paper.}

{\em 
When $\mathcal{J}(t)=1$, one can avoid computation of $h^{(t+1)}(B)$ for all $B \in \mathcal{B}$  in Step~$2$ of GPL. Instead, the fusion center can update only  $h^{(t)}(B(t))$, $h^{(t)}(B_{-j_t}(t-1),1)$ and $h^{(t)}(B_{-j_t}(t-1),0)$ at time $t$, since only these iterates are  required in the Gibbs sampling.}


\subsection{Convergence of GPL}
{\em We will first list a few assumptions that will be crucial in the convergence proof of GPL.}

\begin{assumption}\label{assumption:Lipschitz-continuity-wrt-theta}
The distribution $p_{\theta}(\cdot)$ and the mapping $\mu_2$   as defined before are Lipschitz continuous in $\theta \in \Theta$.
\end{assumption}

\begin{assumption}\label{assumption:fixed-decoding-strategy-for-iid-data}
$\mu_2$ is known to the fusion center.
\end{assumption}

{\em Assumption~\ref{assumption:fixed-decoding-strategy-for-iid-data} allows us to focus only on the sensor subset selection problem rather than the problem of estimating the process given the sensor observations. }

\begin{assumption}\label{assumption:existence-of-lambda*}
Let us consider $Y_B(t)$ with $\theta(t)=\theta \forall t \in \{0,1,2,\cdots\}$ fixed in GPL. Suppose that, one uses   BG to solve the unconstrained problem \eqref{eqn:centralized-unconstrained-problem} for a given $\lambda$, but with the MSE $||X(t)-\hat{X}(t)||^2$ replaced by $Y_{B(t)}(t)$ (under a fixed $\theta$) in the objective function of \eqref{eqn:centralized-unconstrained-problem}, and then finds the $\lambda^*(\theta)$ as in Theorem~\ref{theorem:relation-between-constrained-and-unconstrained-problems} to meet the constraint $\bar{N}$ with equality. We assume that, 
for the given $\beta$ and $\bar{N}$, and for each $\theta \in \Theta$, there exists $\lambda^*(\theta) \in [0,A_0)$ such that,  the optimal Lagrange multiplier to relax this new unconstrained problem is $\lambda^*(\theta)$ (Theorem~\ref{theorem:relation-between-constrained-and-unconstrained-problems}). Also, $\lambda^*(\theta)$ is Lipschitz continuous in $\theta \in \Theta$.
\end{assumption}

{\em 
Assumption~\ref{assumption:existence-of-lambda*} makes sure that the $\lambda(t)$ iteration \eqref{eqn:lambda-update-iid-data} converges, and the   constraint is met with equality.

Let us define the function $\bar{\Gamma}_{\theta}(\phi):=\lim_{\delta \downarrow 0} \frac{[\theta+\delta \phi]_{\Theta}-\theta}{\delta}$; this function is parameterized by $\theta \in \Theta$, and calculates the gradient of the projection function at $\theta$.
}

\begin{assumption}\label{assumption:finite-number-of-local-maximum}
Consider the  function $g(\theta)= \mathbb{E}_{X(t) \sim p_{\theta_0}(\cdot), B(t)=[1,1,\cdots,1]}\log p (Z(t) | \theta)$; this is the expected conditional log-likelihood function of $Z(t)$ conditioned on $\theta$, given that $X(t) \sim p_{\theta_0}(\cdot)$ and $B(t)=[1,1,\cdots,1]$. We assume    that the ordinary differential equation (ODE) $\dot{\theta}(\tau)=\bar
{\Gamma}_{\theta(\tau)}(\nabla g(\theta(\tau)))$ has a globally  asymptotically stable solution $\theta^*$ in the interior of $\Theta$. Also,  $\nabla g(\theta)$ is Lipschitz continuous in $\theta$.
\end{assumption}

{\em 
One can show that the $\theta(t)$ iteration~\eqref{eqn:theta-update-iid-data} asymptotically tracks the ordinary differential equation (ODE) $\dot{\theta}(\tau)=\nabla g(\theta(\tau))$ inside the interior of $\Theta$. In fact, $\bar{\Gamma}_{\theta(\tau)}(\nabla g(\theta(\tau))=\nabla g(\theta(\tau))$ when $\theta(\tau)$ lies inside the interior of $\Theta$. The globally asymptotically stable equilibrium condition on $\dot{\theta}(\tau)=\bar
{\Gamma}_{\theta(\tau)}( \nabla g(\theta(\tau)))$ is required to make sure that the iteration does not converge to some unwanted point on the boundary of $\Theta$ due to the forced projection. The assumption on $\theta^*$ makes sure that the $\theta(t)$ iteration converges to $\theta^*$. However, if there does not exist such a globally asymptotically stable equilibrium, $\theta(t)$ converges almost surely to the set of stationary points of the ODE.

The following result tells us that   the iterates of GPL  almost surely converge to the desired values.
}

\begin{theorem}\label{theorem:convergence-of-GLEM}
Under Assumptions~\ref{assumption:Lipschitz-continuity-wrt-theta},  \ref{assumption:fixed-decoding-strategy-for-iid-data},  \ref{assumption:existence-of-lambda*},  \ref{assumption:finite-number-of-local-maximum} and the GPL algorithm, we have  $\lim_{t \rightarrow \infty}\theta(t)=\theta^*$ almost surely. Correspondingly, $\lambda(t) \rightarrow \lambda^*(\theta^*)$ almost surely. 
Also, $f^{(t)}(B) \rightarrow\mathbb{E}_{Z_B(t) \sim p(\cdot| \theta_0)} \mathbb{E}_{X(t) \sim p(\cdot|\theta^*,Z_B(t))}(||X(t)-\hat{X}_B(t)||^2|Z_B(t), \theta^*) =:f_{\theta^*}(B)$ almost surely for all $B \in \mathcal{B}$.  The $B(t)$ process reaches the steady-state distribution  $\pi_{\beta,f_{\theta^*},\lambda^*(\theta^*), \theta^*}(\cdot)$ which can be obtained by replacing $h(B)$ in \eqref{eqn:definition-of-Gibbs-distribution} by $f_{\theta^*}(B)+\lambda^*(\theta^*)||B||_1$ where $f_{\theta^*}(B)$ is the MSE under configuration $B$ if the true parameter is $\theta^*$. 

In case there does not exist a globally asymptotically stable equilibrium $\theta^*$, the   $\theta(t)$ iteration almost surely converges to the stationaly points of the ODE $\dot{\theta}(\tau)=\bar
{\Gamma}_{\theta(\tau)}(\nabla g(\theta(\tau)))$.
\end{theorem}

\begin{proof}
See Appendix~\ref{appendix:proof-of-GLEM}.
\end{proof}

{\em Now we make a few observations. If $\theta(t) \rightarrow \theta^*$, but the constraint in \eqref{eqn:centralized-constrained-problem} is satisfied  with $\lambda=0$ and  policy $\mu_2(\cdot; \cdot; \theta^*)$,\footnote{The other two arguments of $\mu_2$ are $B(t)$ and $Z_{B(t)}(t)$} then  $\lambda(t) \rightarrow 0$, i.e., $\lambda^*(\theta^*)=0$, and the constraint becomes redundant.   If $\theta^*$ exists, then, under the above assumptions, we will have $\theta^*=\theta_0$, and   the algorithm reaches the global optimum.}

{\em 
If all sensors are not read when $\mathcal{J}(t)=1$, then one has to update $\theta(t)$ based on the observations   $Z_{B(t)}(t)$ collected from the sensors determined by $B(t)$. In that case, $\theta(t)$ will converge to a stationary point $\theta_1$ of $g_1(\theta):=\lim_{t \rightarrow \infty}\mathbb{E}_{X(t) \sim p_{\theta_0}(\cdot),B(t) \sim \pi_{\beta, f_{\theta}, \lambda^*(\theta), \theta}}  \log p(Z_{B(t)}(t)|\theta)$, which will be different from $\theta^*$ of Theorem~\ref{theorem:convergence-of-GLEM} in general. However, in the numerical example   in Section~\ref{section:numerical-results}, we observe numerically that $\theta_1=\theta^*$ can be possible.}

\section{Numerical Results}\label{section:numerical-results}
\subsection{Performance of BG~algorithm}
\label{subsection:numerical-performance-of-basic-gibbs-sampling}

{\em 
For the sake of illustration, we consider $N=10$~sensors which are supposed to sense $X=\{X_1,X_2,\cdots,X_{10}\}$, where 
$X$ is a jointly Gaussian random vector with covariance matrix $M$. Sensor~$k$ has access only to $X_k$. The matrix $M$ is chosen as follows. We generated a random $N \times N$ matrix $A$ whose elements are uniformly and independently distributed over the interval $[-1,1]$, and set $M=A^T A$ as the covariance matrix of $X$. We set sensor activation cost $\lambda=2$, and seek to solve \eqref{eqn:centralized-unconstrained-problem}. We assume that sensing at each node is perfect,  and 
that the fusion center estimates $\hat{X}$ from the observation $\{X_i\}_{i \in S}=:X_S$ as $\mathbb{E}(X | X_S)$, where $S$ is the set of active sensors. Under such an estimation scheme, the conditional distribution of 
$X_{S^c}$ is still a jointly Gaussian random vector with mean $\mathbb{E}(X_{S^c} | X_S)$ and the covariance matrix $M(S^c,S^c)-M(S^c,S) M(S,S)^{-1}M(S,S^c)$ (see \cite[Proposition~$3.4.4$]{hajek-lecture-note}), where $M(S,S^c)$ is the restriction of $M$ to the rows indexed by $S$ and the columns indexed by $S^c$. The trace of this covariance matrix gives the MMSE when the subset $S$ of sensors are active. }

\begin{figure}[!t]
\begin{center}
\includegraphics[height=5cm, width=7cm]{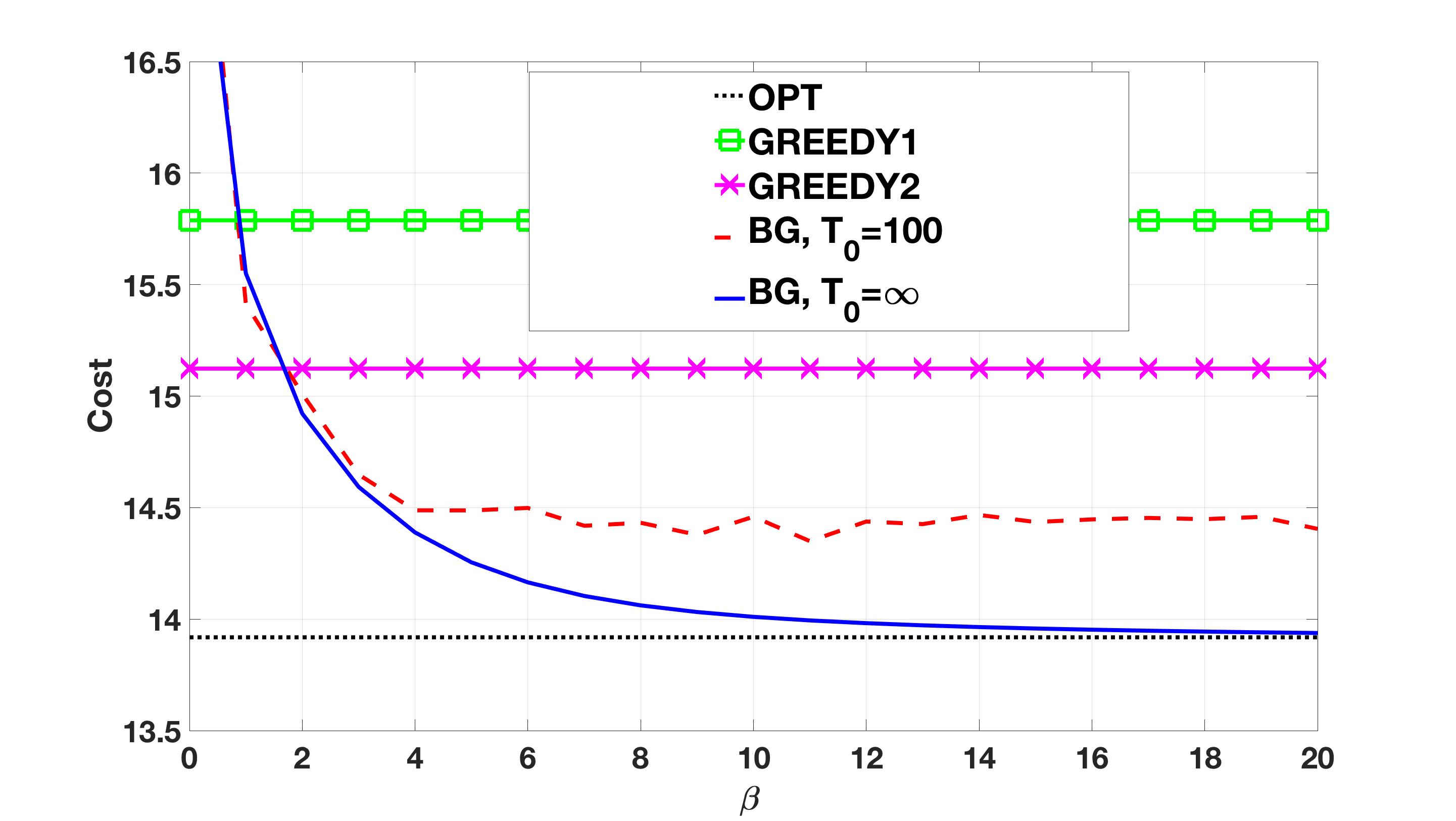}
\end{center}
\caption{Comparison among OPT, BG under steady state ($T_0=\infty$),  BG with finite iterations, GREEDY1 and GREEDY2, for solving problem~\eqref{eqn:centralized-unconstrained-problem}.   For each $\beta$, BG with finite iterations stops after $T_0=100$~iterations. The results of BG with finite iterations are averaged over $100$ independent sample paths. Details are provided in Section~\ref{subsection:numerical-performance-of-basic-gibbs-sampling}.}
\label{fig:comparison--gibbs-optimal-greedy-finitegibbs}
\end{figure}

{\em 
In Figure~\ref{fig:comparison--gibbs-optimal-greedy-finitegibbs}, we compare the cost for    five algorithms:
\begin{itemize}
\item OPT: Here we consider the minimum   cost for \eqref{eqn:centralized-unconstrained-problem}.
\item BG under steady state: Here   the configuration $B \in \mathcal{B}$ is chosen 
according to the distribution $\pi_{\beta}(\cdot)$ defined in Section~\ref{section:gibbs-sampling-unconstrained-problem}, which can be obtained by running BG for $T_0=\infty$ iterations. 
This is done for several values of $\beta$.
\item BG  with finite iteration: Here we run BG~algorithm for $T_0=100$~iterations. This is done independently for several values of $\beta$, where for each $\beta$ the iteration starts from an independent random configuration. Note that, we have simulated $100$ independent sample paths of BG for each $\beta$, and averaged the result over these sample paths. 
\item GREEDY1: Start with an empty set $S$, and find the cost if this subset of sensors are activated. Then compare this cost with the cost in case sensor~$1$ is added to this set. If it turns out that adding sensor~$1$ to this set $S$ reduces the cost, then add sensor~$1$ to the set $S$; otherwise, remove sensor~$1$ from set $S$. Do this operation serially for all sensors, and activate the  sensors given by the final set $S$.
\item GREEDY2: Start with an empty set $S$, and find the cost if this subset of sensors are activated. Then find the sensor~$j_1$ which, when added to $S$, will result in the minimum cost. If the cost for $S \cup \{j_1\}$ is less than that of $S$, then do $S=S \cup \{j_1\}$.  Now find the sensor~$j_2$ which, when added to $S$, will result in the minimum cost. If the cost for $S \cup \{j_2\}$ is less  than that of $S$, then do $S=S \cup \{j_2\}$.  Repeat this operation $N$ times, and activate the set of  sensors given by the final set $S$. This algorithm is adapted from \cite{wang-etal16efficient-observation-selection}.
\end{itemize}
}

\begin{figure}[!t]
\begin{center}
\includegraphics[height=5cm, width=7cm]{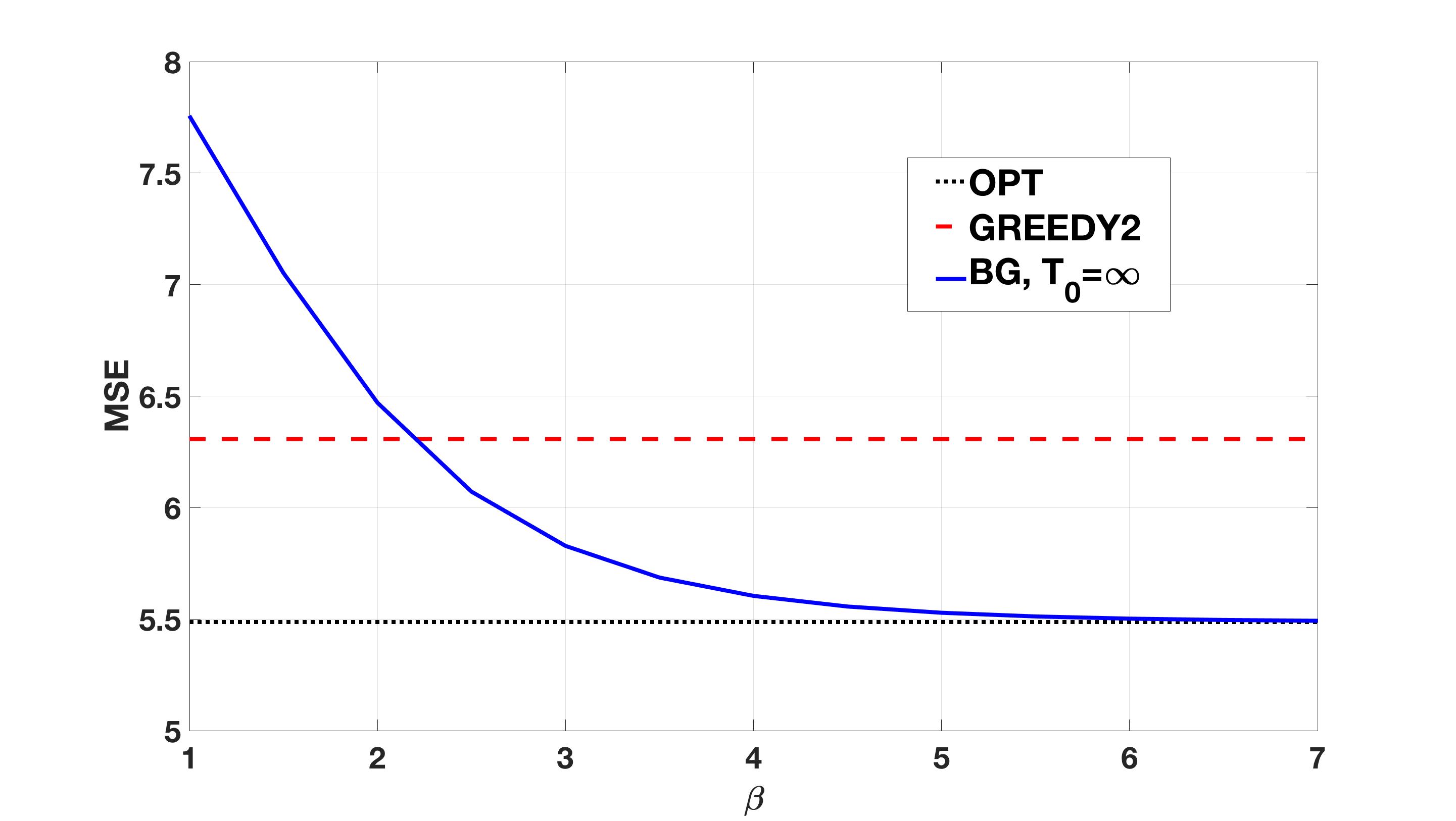}
\end{center}
\vspace{-6mm}
\caption{Comparison among OPT, BG under steady state ($T_0=\infty$), and GREEDY2, for solving problem~\eqref{eqn:constrained-optimization-problem-static-data-parametric-distribution}.   Details are provided in  Section~\ref{subsection:numerical-gibbs-sampling-applied-to-hard-constrained-problem}.}
\label{fig:gibbs_optimal_greedy_comparison_fixed_number_of_active_sensors}
\end{figure}

{\em It turns out that, under the optimal configuration, $5$ sensors are activated and the optimal cost is $13.9184$. GREEDY1 activates $7$~sensors and incurred a cost of $15.7881$. On the other hand, GREEDY2 activates $6$~sensors and incurs a cost of $15.1234
$. However, we are not aware of any monotonicity or supermodularity property of the objective function in \eqref{eqn:centralized-unconstrained-problem}; hence, we cannot provide any constant approximation ratio guarantee for GREEDY1 and GREEDY2 algorithms for the problem~\eqref{eqn:centralized-unconstrained-problem}. On the other hand, we have already proved that BG performs near optimally for large $\beta$. Hence, we choose to investigate the performance of BG, though it might require more number of iterations compared to $N=10$ iterations for GREEDY1 or $O(\frac{N(N-1)}{2})$ iterations for GREEDY2. It is important to note that, \eqref{eqn:centralized-unconstrained-problem} is NP-hard, and BG  allows us to avoid searching over $2^N$ possible configurations. }

{\em In Figure~\ref{fig:comparison--gibbs-optimal-greedy-finitegibbs}, we can see that for $\beta \geq 3$, the steady state distribution $\pi_{\beta}(\cdot)$ of BG achieves better expected cost than GREEDY1 and GREEDY2, and the cost becomes closer to the optimal cost as $\beta$ increases. On the other hand, for each $\beta \geq 5$, BG after $100$~iterations yielded a configuration   that achieves near-optimal cost. Hence, BG with reasonably  small number of iterations can be used to find the optimal subset of active sensors. Note that, in this numerical example, BG with $100$~iterations need to compute the cost for $200$~configurations, while GREEDY1 and GREEDY2 need to compute the cost for $10$ and $45$ configurations respectively; but this additional amount of computation (which is much less that $2^N$) can significantly reduce the cost. However, the real advantage of Gibbs sampling based subset selection over the greedy subset selection is that, when an unknown distribution is learnt over time, the greedy algorithm has to be re-run each time with $O(N)$ or $O(N^2)$ complexity, whereas Gibbs sampling can be run in each slot iteratively with minimal computational cost while achieving near-optimal performance under   large $\beta$.}

\begin{figure}[!t]
\begin{center}
\includegraphics[height=4cm, width=6cm]{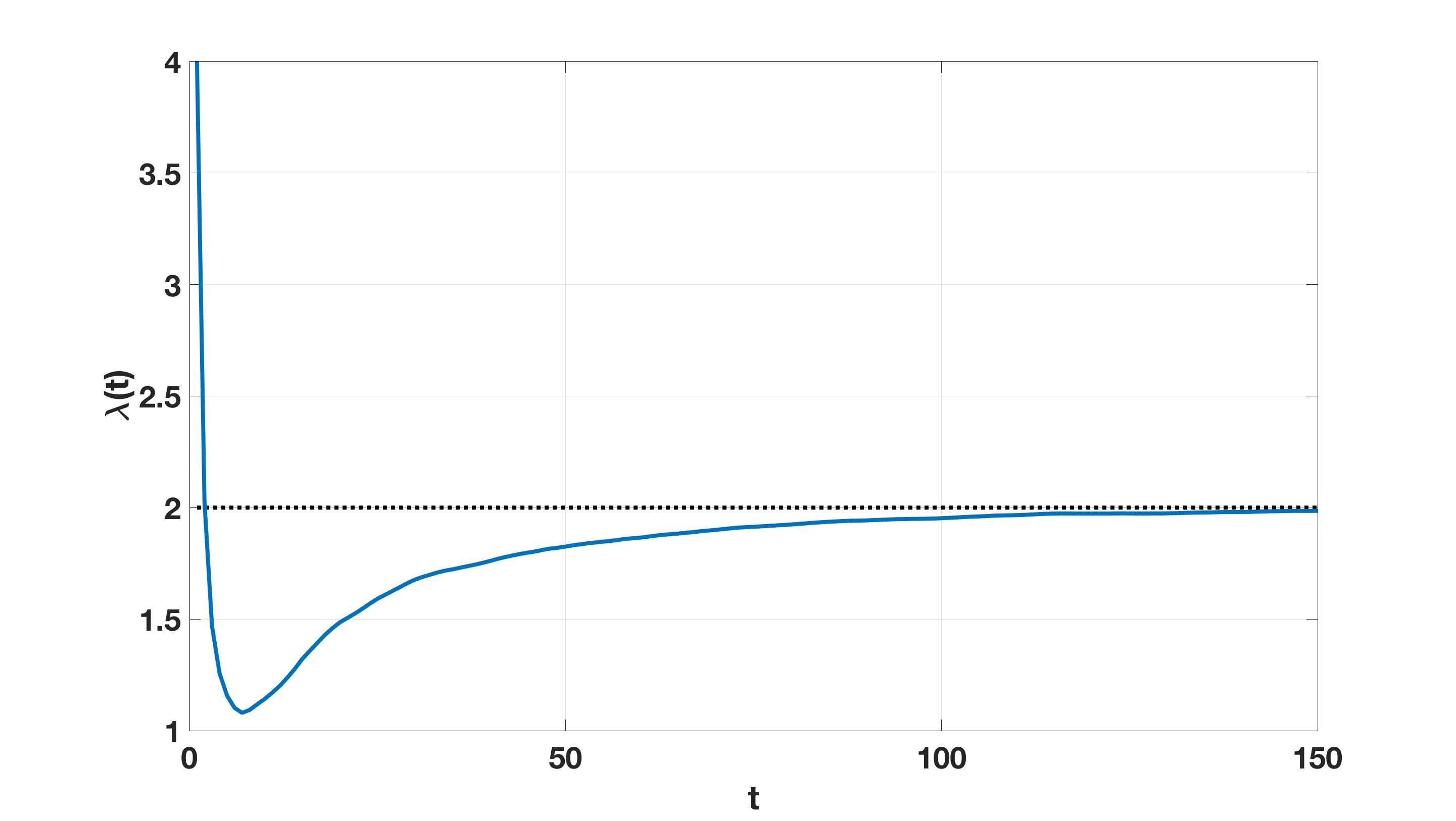}
\end{center}
\vspace{-4mm}
\caption{Illustration for convergence speed of $\lambda(t)$ (averaged over $50$ independent sample paths) in the GL~algorithm.  }
\label{fig:gibbs-stochastic-approximation}
\end{figure}

\begin{figure*}[t]
 \begin{minipage}[r]{0.5\linewidth}
\subfigure{
\includegraphics[height=5cm, width=\linewidth]{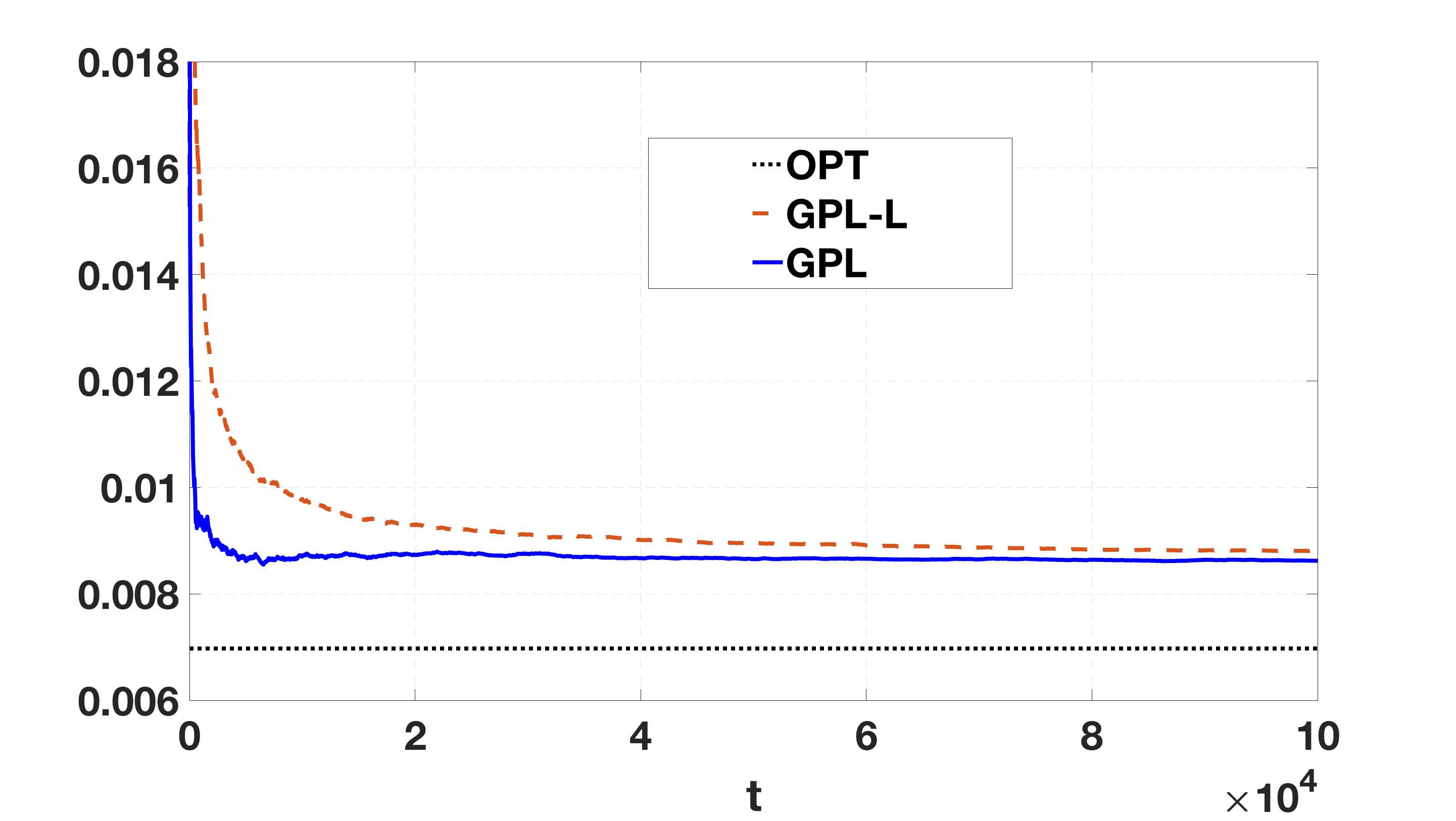}
\includegraphics[height=5cm, width=\linewidth]{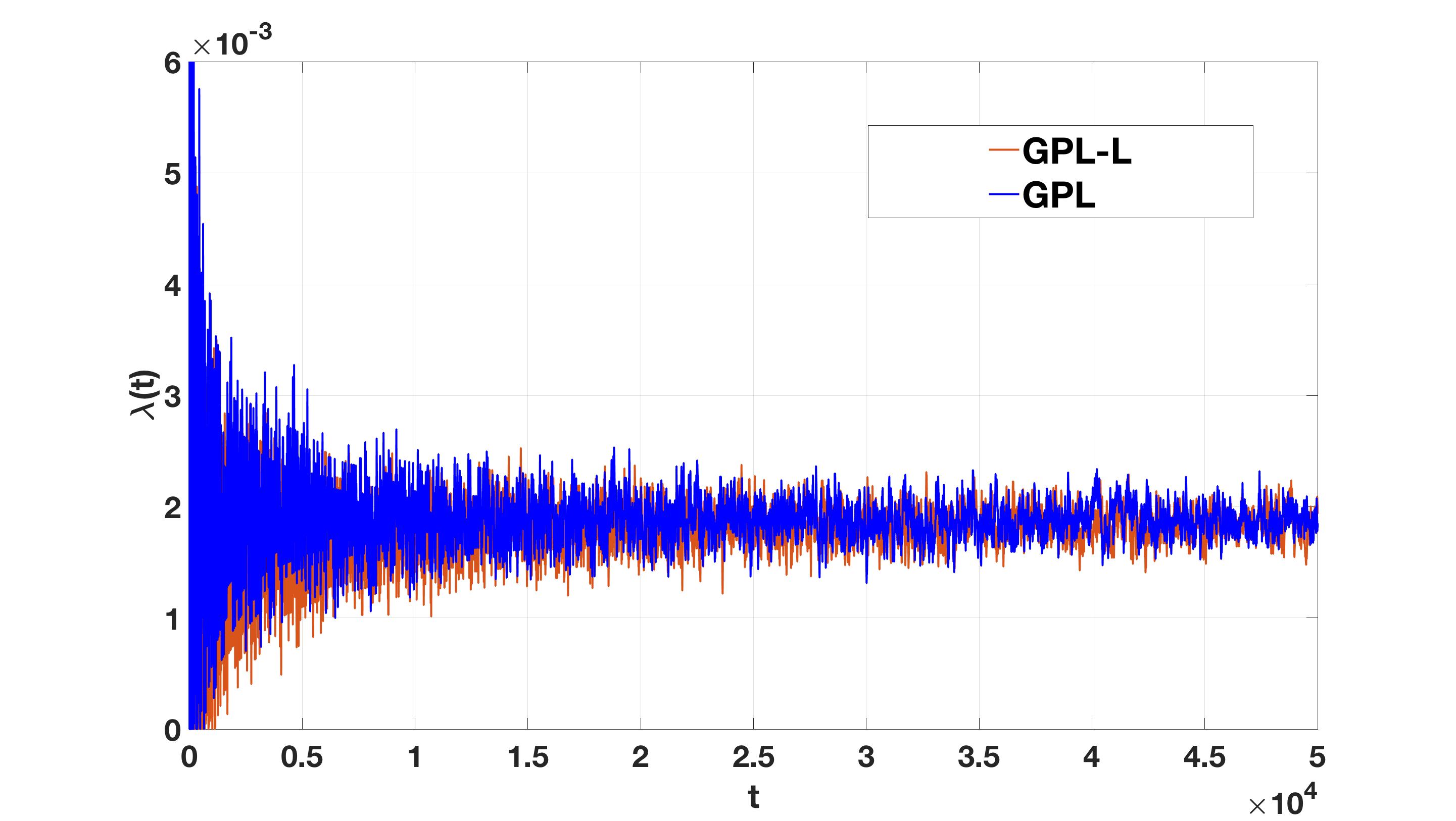}}
\end{minipage}  \hfill
\end{figure*}
 \begin{figure*}[t]
\begin{minipage}[c]{0.5\linewidth}
\subfigure{
\includegraphics[height=5cm, width=\linewidth]{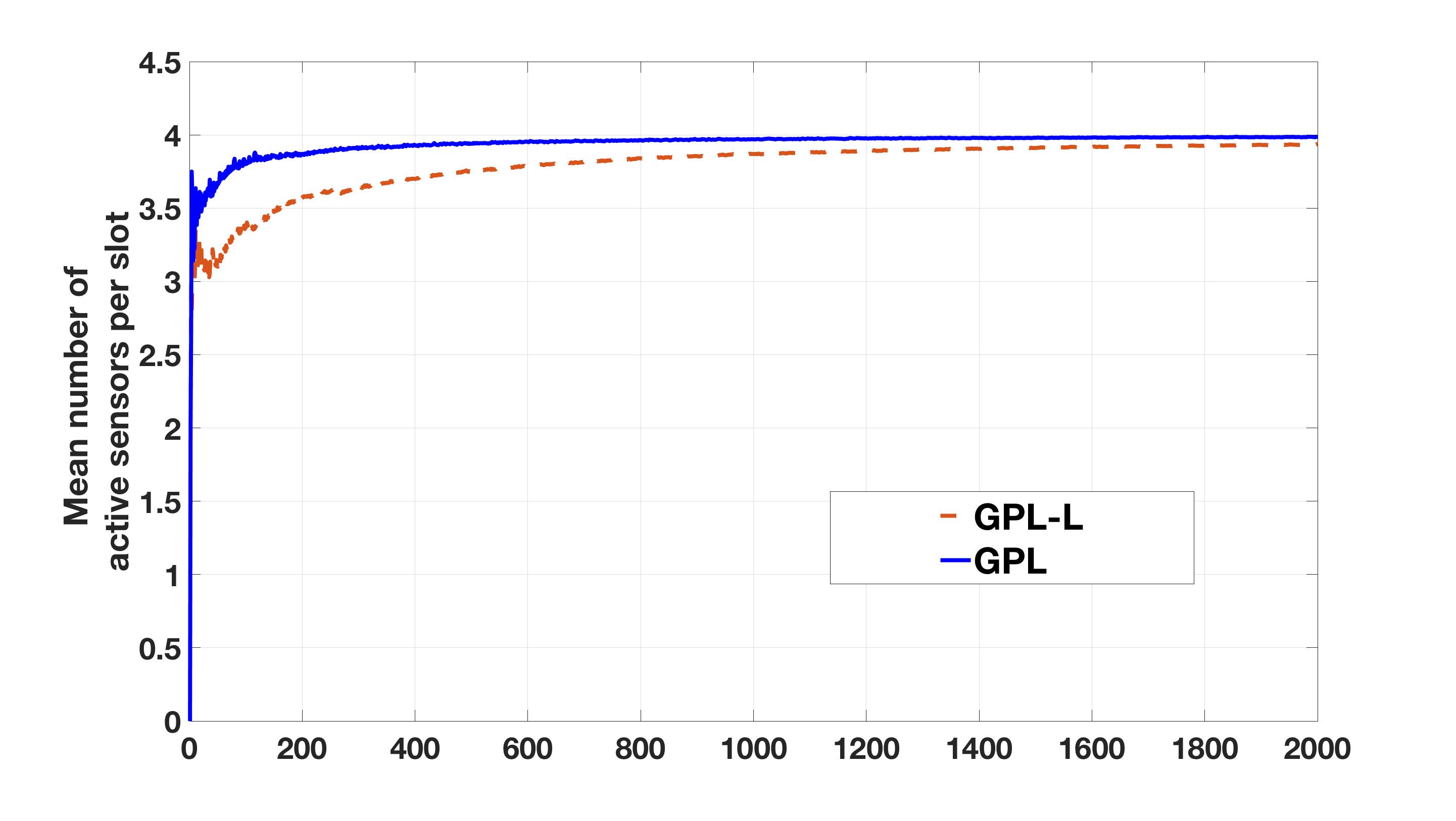}
\includegraphics[height=5cm, width=\linewidth]{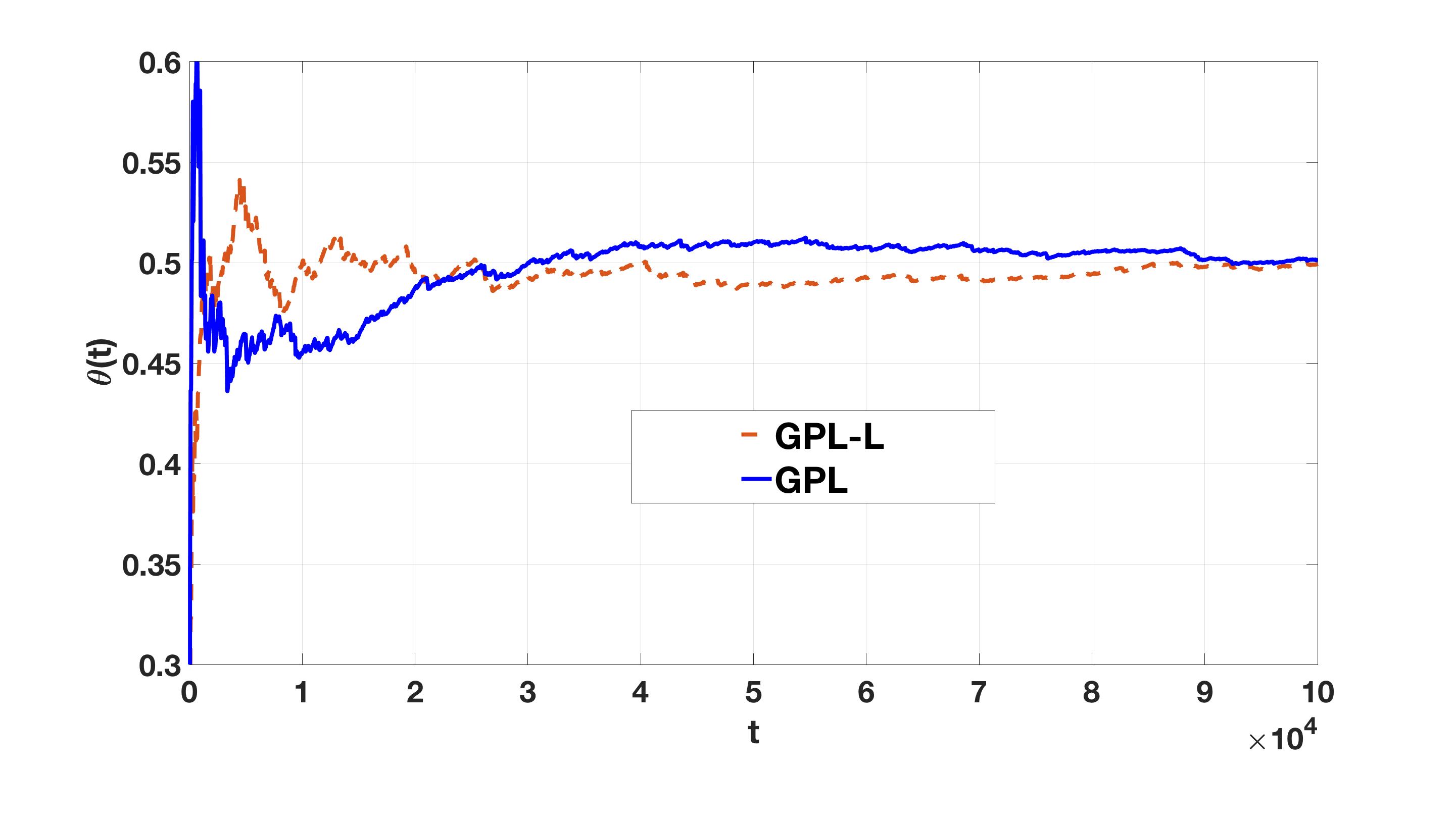}}
\end{minipage} \hfill
\caption{Performance of GPL for centralized tracking of the i.i.d process.}
\label{fig:performance-of-centralized-iid-tracking}
\end{figure*}

\subsection{Performance of BG applied to problem~\eqref{eqn:constrained-optimization-problem-static-data-parametric-distribution}}\label{subsection:numerical-gibbs-sampling-applied-to-hard-constrained-problem}
{\em 
Here we seek to solve problem~\eqref{eqn:constrained-optimization-problem-static-data-parametric-distribution} with $\bar{N}=4$ under the same setting as in Section~\ref{subsection:numerical-performance-of-basic-gibbs-sampling} except that a new sample of  the covariance matrix $M$ is chosen. Here we compare the MSE for the following three cases:
\begin{itemize}
\item  OPT: Here we choose an optimal subset  for \eqref{eqn:constrained-optimization-problem-static-data-parametric-distribution}.
\item  BG  under steady state: Here we assume that the configuration $B$ is chosen 
according to the steady-state distribution $\pi_{\beta}(\cdot)$, but restricted only to the set $\{B \in \mathcal{B}: ||B||_1=\bar{N}\}$. This is done by putting $h(B)=\mathbb{E}||X-\hat{X}||^2$ if $||B||_1 = \bar{N}$ and $h(B)=\infty$ otherwise.
This is done for several values of $\beta$.
\item GREEDY2: Start with an empty set $S$, and find the MSE if this subset of sensors are activated. Then find the sensor~$j_1$ which, when added to $S$, will result in the minimum MSE. If the MSE for $S \cup \{j_1\}$ is less than that of $S$, then do $S=S \cup \{j_1\}$.  Now find the sensor~$j_2$ which, when added to $S$, will result in the minimum MSE. If the MSE for $S \cup \{j_2\}$ is less  than that of $S$, then do $S=S \cup \{j_2\}$.  Repeat this   until we have $|S|=\bar{N}$, and activate the set of $\bar{N}$ sensors given by the final set $S$. 
 A similar greedy algorithm is used in \cite{wang-etal16efficient-observation-selection}.
\end{itemize}
The performances for these three cases are shown in Figure~\ref{fig:gibbs_optimal_greedy_comparison_fixed_number_of_active_sensors}.  BG outperforms GREEDY2 for $\beta \geq 3$, and becomes very close to OPT performance for $\beta \geq 5$. 
}

\subsection{Convergence speed of GL~algorithm}
\label{subsection:numerical-convergence-speed-gibbs-sampling-stochastic-approximation}
{\em We consider a setting similar to that of Section~\ref{subsection:numerical-performance-of-basic-gibbs-sampling}, except that we  fix $\beta=5$, and choose an $M$ which is different from that in Section~\ref{subsection:numerical-performance-of-basic-gibbs-sampling}. Under this setting, for $\lambda^*=2$, BG~algorithm yields the MMSE $2.4303$, and the expected number of sensors activated by BG~algorithm becomes $6.5247$. Now, let us consider problem~\eqref{eqn:centralized-constrained-problem} with the constraint value $\bar{N}=6.5247$. Clearly, if GL~algorithm is employed to find out the solution of problem~\eqref{eqn:centralized-constrained-problem} with $\bar{N}=6.5247$, then $\lambda(t)$ should converge to $\lambda^*=2$. }

{\em The evolution of $\lambda(t)$ (averaged over $50$ independent sample paths) under GL   is shown in  Figure~\ref{fig:gibbs-stochastic-approximation}. We can see that, starting from $\lambda(0)=4$ and and using the stepsize sequence $a(t)=\frac{1}{t}$, the iterate  $\lambda(t)$ becomes very close to $\lambda^*=2$ within $100$~iterations.  
Thus, our numerical illustration shows that GL~algorithm has reasonably fast convergence rate  for practical active sensing. We will later demonstrate convergence of the mean number of active sensors per slot to $\bar{N}$ for GPL, and hence do not show it here.}

\subsection{Performance of GPL}
{\em 
Now we   demonstrate the performance of GPL to solve \eqref{eqn:centralized-constrained-problem}.  We consider the following parameter values: $N=10$,  $\bar{N}=4$, $a(t)=\frac{0.1}{t^{0.6}}$, $b(t)=\frac{0.1}{t^{0.8}}$, $c(t)=\frac{0.1}{t}$, $d(t)=\frac{0.1}{t^{0.1}}$, $T=50$, $\lambda(0)=0.05$, $\beta=1000$. Gibbs sampling is run $10$ times per slot.

For illustration purpose, we assume that $X(t)  \sim \mathcal{N}(\theta_0, (1-\theta_0)^2)$ {\em scalar}, and $z_k(t)=X(t)+w_k(t)$, where $\theta_0=0.5$  and $w_k(t)$ is zero mean  i.i.d. Gaussian noise independent across $k$. Standard deviation of $w_k(t)$ is chosen uniformly and independently from the interval $[0,0.5]$, for each $k \in \{1,2,\cdots,N\}$.  Initial estimate $\theta(0)=0.2$, $\Theta=[0,0.8]$. 

We consider three possible algorithms and cases: (i) GPL in its basic form  (all sensors are read when $\mathcal{J}(t)=1$, and $\theta(t)$ and $f^{(t)}(B)$ are updated for all $B \in \mathcal{B}$ when $\mathcal{J}(t)=1$), (ii) a low-complexity  variation of GPL called GPL-L where all sensors are not read when $\mathcal{J}(t)=1$, and  $f^{(t)}(B(t))$ and $\theta(t)$ updates are done  every $T$ slots, and (iii) the  OPT case where   $\bar{N}$ sensors with  smallest observation noise variances are used for MMSE estimation in each slot,  with a perfect knowledge of $\theta_0=0.5$.  

The time-average MSE per slot, mean number of active sensors per slot, $\lambda(t)$ and $\theta(t)$ are plotted against   $t$ in Figure~\ref{fig:performance-of-centralized-iid-tracking}.  MSE  of all these three algorithms are much smaller than  $Var(X(t))=(1-\theta_0)^2$ (this is MMSE  without any observation). We notice that GPL and GPL-L perform close to OPT in terms of time-average MSE; this shows the power of Gibbs sampling and learning $\theta(t)$ over time. We also observe that, GPL converges faster than GPL-L, at the expense of additional computation and communication; but both algorithms asymptotically offer the same MSE per unit time.  We have plotted only one sample path since  the algorithms  converge almost surely to the global optimum in this case, as observed in the simulation. We observe that 
$\frac{1}{t}\sum_{\tau=1}^t ||B(\tau)||_1 \rightarrow \overline{N}$ and $\theta(t) \rightarrow \theta_0$ almost surely 
for both algorithms (verified by  simulating  multiple sample paths). It is interesting to note that $\theta^*=\theta_1=\theta_0$ in this numerical example (recall Theorem~\ref{theorem:convergence-of-GLEM} and the observations after that), i.e., both algorithms  converge to the true parameter value $\theta_0$. Convergence rate   will vary with stepsize   and other parameters, and hence is not discussed here. 

Note that, we have already shown performance improvement by the use of BG against GREEDY1 and GREEDY2 algorithms for known distribution; see Section~\ref{subsection:numerical-performance-of-basic-gibbs-sampling}. Hence, we do not consider asymptotic performance improvement of GPL against those two algorithms. \footnote{However, it is important to note that the OPT performance for the specific numerical example in Figure~\ref{fig:performance-of-centralized-iid-tracking} can be achieved by GREEDY2 because of the simple model and known $\theta_0$, but this is not true in general as observed in the numerical example of Section~\ref{subsection:numerical-performance-of-basic-gibbs-sampling}. In order to use GREEDY1 and GREEDY2 when $\theta_0$ is unknown, one has to run these two algorithms each time $\theta(t)$ is updated, and for this the MSE for all $B \in \mathcal{B}$ need to be recomputed for each new value of $\theta(t)$. On the contrary, performing one or a few steps of Gibbs sampling will be much easier in a particular slot.}

}

\section{Conclusions}\label{section:conclusion}
{\em 
We have proposed low-complexity centralized   learning algorithms for dynamic sensor subset selection for tracking i.i.d. time-varying   processes. We first provided  algorithms based on Gibbs sampling and stochastic approximation for i.i.d. time-varying data with known  distribution, and later provided learning algorithms for unknown, parametric distribution, and proved almost sure convergence.    Numerical results demonstrate the efficacy of the algorithms against simple algorithms without learning. In future, we seek to develop distributed tracking algorithms for i.i.d. process and Markov chains with known and unknown dynamics.
}

{\small
\bibliographystyle{unsrt}
\bibliography{arpan-techreport}
}

%
%
%

%

\appendices

{\em

\section{Proof of Theorem~\ref{theorem:relation-between-constrained-and-unconstrained-problems}}
\label{appendix:proof-of-relation-between-constrained-and-unconstrained-problems}
We will prove only the first part of the theorem where there exists one $B^*$. The second part of the theorem can be proved similarly. Let us denote the optimizer for \eqref{eqn:centralized-constrained-problem} by $B$, which is possibly different from $B^*$. Then, by the definition of $B^*$, we have $f(B^*)+ \lambda^* ||B^*||_1 \leq f(B)+ \lambda^* ||B||_1$. But $||B||_1 \leq K$ (since $B$ is a feasible solution to the constrained problem) and $||B^*||_1=K$ (by assumption). Hence, $f(B^*) \leq f(B)$. This completes the proof.

\section{Weak and Strong Ergodicity}
\label{appendix:weak-and-strong-ergodicity}
Consider a discrete-time  Markov chain (possibly not time-homogeneous) $\{B(t)\}_{t \geq 0}$ with transition probability matrix (t.p.m.)  $P(m;n)$ between 
$t=m$ and $t=n$. We denote by $\mathcal{D}$   the collection of all possible probasbility distributions 
 on the state space. Let $d_V(\cdot,\cdot)$ denote  the total variation distance between two distributions in $\mathcal{D}$. 
Then $\{B(t)\}_{t \geq 0}$ is called weakly ergodic if, for all $m \geq 0$,  we have 
$\lim_{n \uparrow \infty} \sup_{\mu,\nu \in \mathcal{D}}  d_V (\mu P(m;n) , \nu P(m;n) ) =0 $.  The Markov chain $\{B(t)\}_{t \geq 0}$ is called strongly ergodic if there exists $\pi \in \mathcal{D}$ such that, 
$\lim_{n \uparrow \infty} \sup_{\mu \in \mathcal{D}}  d_V (\mu^{T} P(m;n) , \pi ) =0 $ for all $m \geq 0$.

\section{Proof of Theorem~\ref{theorem:result-on-weak-and-strong-ergodicity}}
\label{appendix:proof-of-weak-and-strong-ergodicity}
 
We will first show that the Markov chain $\{B(t)\}_{t \geq 0}$ in weakly ergodic.

Let us define $\Delta:=\max_{B \in \mathcal{B}, A \in \mathcal{B}}|h(B)-h(A)|$.

Consider the transition probability matrix (t.p.m.) $P_l$ for the inhomogeneous Markov 
chain $\{Y(l)\}_{l \geq 0}$ (where $Y(l):=B(lN)$). The Dobrushin's ergodic coefficient $\delta(P_l)$ is given by 
(see \cite[Chapter~$6$, Section~$7$]{breamud99gibbs-sampling} for definition) 
$\delta(P_l)=1- \inf_{B^{'},B^{''} \in \mathcal{B}} \sum_{B \in \mathcal{B}} \min \{P_l(B^{'},B),P_l(B^{''},B) \}$. 
A sufficient condition for the Markov chain $\{B(t)\}_{t \geq 0}$  
to be weakly ergodic is $\sum_{l=1}^{\infty}(1-\delta(P_l))=\infty$ (by 
\cite[Chapter~$6$, Theorem~$8.2$]{breamud99gibbs-sampling}).

Now, with positive probability, activation states for all nodes are updated over a period of $N$ slots. Hence, $P_l(B^{'},B)>0$ for all $B^{'},B \in \mathcal{B}$. Also, once a node $j_t$ for $t=lN+k$ 
is chosen  in ABG~algorithm, the sampling probability for any activation state in a slot 
is greater than  $\frac{e^{-\beta(lN+k) \Delta}}{2}$. 
Hence, for independent sampling over $N$ slots, we  have, for all pairs $B^{'},B$: 
$$P_l(B^{'},B) >  \prod_{k=0}^{N-1}\bigg( \frac{e^{-\beta(lN+k) \Delta}}{2N} \bigg) >0$$ 
Hence, 
\begin{eqnarray}
&&\sum_{l=0}^{\infty}(1-\delta(P_l)) \nonumber\\
&=& \sum_{l=0}^{\infty} \inf_{B^{'},B^{''} \in \mathcal{B}} \sum_{B \in \mathcal{B}} \min \{P_l(B^{'},B),P_l(B^{''},B) \} \nonumber\\
& \geq & \sum_{l=0}^{\infty} 2^N \prod_{k=0}^{N-1} \bigg( \frac{e^{-\beta(0) \log(1+lN+k) \times \Delta}}{2N} \bigg) \nonumber\\
& \geq &   \sum_{l=0}^{\infty}  \prod_{k=0}^{N-1} \bigg( \frac{e^{-\beta(0) \log(1+lN+N) \times \Delta}}{N} \bigg)  \nonumber\\
& = &  \frac{1}{N^N} \sum_{l=1}^{\infty}  \frac{1}{  (1+lN)^{\beta(0) N \Delta}}  \nonumber\\
& \geq &  \frac{1}{ N^{N+1}} \sum_{i=N+1}^{\infty}  \frac{1}{  (1+i)^{\beta(0) N \Delta}}  \nonumber\\
& = & \infty
 \end{eqnarray}
Here the first inequality uses the fact that the cardinality of $\mathcal{B}$ is $2^N$. The second inequality follows from  replacing $k$ by $N$ in the numerator. The third inequality follows from lower-bounding  
$\frac{1}{(1+lN)^{\beta(0) N \Delta}}$ by $\frac{1}{N}\sum_{i=lN}^{lN+N-1}  \frac{1}{  (1+i)^{\beta(0) N \Delta}} $. 
The last equality follows from the fact that $\sum_{i=1}^{\infty} \frac{1}{i^a}$ diverges for $0 <a<1$.

 Hence, the Markov chain  $\{B(t)\}_{t \geq 0}$ is  weakly ergodic.
 
 In order to prove strong ergodicity of 
 $\{B(t)\}_{t \geq 0}$, we invoke \cite[Chapter~$6$, Theorem~$8.3$]{breamud99gibbs-sampling}. 
 We denote the  t.p.m. of $\{B(t)\}_{t \geq 0}$ at a specific time $t=T_0$ 
 by $Q^{(T_0)}$, which is a given specific matrix. If  $\{B(t)\}_{t \geq 0}$   evolves up to infinite time 
 with {\em fixed} t.p.m. $Q^{(T_0)}$, then it will reach the stationary distribution $\pi_{\beta_{T_0}}(B)= \frac{e^{-\beta_{T_0} h(B)}}{Z_{\beta_{T_0}}}$. 
 Hence, we can claim that Condition~$8.9$ of \cite[Chapter~$6$, Theorem~$8.3$]{breamud99gibbs-sampling} is  satisfied. 
 
 Next, we check Condition~$8.10$ of \cite[Chapter~$6$, Theorem~$8.3$]{breamud99gibbs-sampling}. 
 For any $B \in \arg \min_{B^{'} \in \mathcal{B}} h(B^{'})$,  we can argue that $\pi_{\beta_{T_0}}(B)$ increases with $T_0$ for 
 sufficiently large $T_0$; this can be verified by considering the derivative of $\pi_{\beta}(B)$ w.r.t. $\beta$. For $B \notin \arg \min_{B^{'} \in \mathcal{B}} h(B^{'})$,  the probability 
 $\pi_{\beta_{T_0}}(B)$ decreases with $T_0$ for large $T_0$. Now, using the fact that any monotone, bounded sequence converges, we can write $\sum_{T_0=0}^{\infty} \sum_{B \in \mathcal{B}} |\pi_{\beta_{T_0+1}}(B)-\pi_{\beta_{T_0}}(B)| < \infty$. 
 
 Hence, by \cite[Chapter~$6$, Theorem~$8.3$]{breamud99gibbs-sampling}, the Markov chain $\{B(t)\}_{t \geq 0}$ is strongly ergodic.  It  is straightforward to verify the claim regarding the  limiting distribution.

\section{Proof of Lemma~\ref{lemma:active-sensors-decreasing-in-lambda}}
\label{appendix:proof-of-active-sensors-decreasing-in-lambda}
Let $\lambda_1 > \lambda_2 > 0$, and the corresponding optimal error and mean number of active sensors under these multiplier values be $(f_1,n_1)$ and $(f_2,n_2)$, respectively. Then, by definition, $f_1 + \lambda_1 n_1 \leq f_2 + \lambda_1 n_2$ and $f_2 + \lambda_2 n_2 \leq f_1 + \lambda_2 n_1$. Adding these two inequalities, we obtain $\lambda_1 n_1+\lambda_2 n_2 \leq \lambda_1 n_2+\lambda_2 n_1$, i.e., $(\lambda_1-\lambda_2)n_1 \leq (\lambda_1-\lambda_2)n_2$. Since $\lambda_1>\lambda_2$, we obtain $n_1 \leq n_2$. This completes the first part of the proof. The second part of the proof follows using similar arguments.

\section{Proof of Lemma~\ref{lemma:active-sensors-decreasing-in-lambda-under-basic-gibbs-sampling}}
\label{appendix:proof-of-active-sensors-decreasing-in-lambda-under-basic-gibbs-sampling}
Let us denote $\mathbb{E}_{\mu_2}||B(t)||_1=:g(\lambda)=\frac{\sum_{B \in \mathcal{B}} ||B||_1 e^{-\beta h(B)}}{Z_{\beta}}$. 
It is straightforward to see that $\mathbb{E}_{\mu_2}||B(t)||_1$ is continuously differentiable in $\lambda$.  
Let us denote $Z_{\beta}$ by $Z$ for simplicity. 
The derivative of $g(\lambda)$ w.r.t. $\lambda$ is given by:
\tiny
\begin{eqnarray*}
&& g'(\lambda)\\
&=& \frac{  -Z \beta \sum_{B \in \mathcal{B}} ||B||_1^2 e^{-\beta(f(B)+\lambda ||B||_1)} - \sum_{B \in \mathcal{B}} ||B||_1 e^{-\beta(f(B)+\lambda ||B||_1)}  \frac{dZ}{d \lambda}        }{Z^2}
\end{eqnarray*}
\normalsize

Now, it is straightforward to verify that $\frac{dZ}{d \lambda}=-\beta Z g(\lambda)$. Hence, 
\tiny
\begin{eqnarray*}
&& g'(\lambda)\\
&=& \frac{  -Z \beta \sum_{B \in \mathcal{B}} ||B||_1^2 e^{-\beta(f(B)+\lambda ||B||_1)} + \sum_{B \in \mathcal{B}} ||B||_1 e^{-\beta(f(B)+\lambda ||B||_1)} \beta Z g(\lambda)       }{Z^2}
\end{eqnarray*}
\normalsize

Now, $g'(\lambda) \leq 0$ is equivalent to 
$$g(\lambda) \leq \frac{  \sum_{B \in \mathcal{B}} ||B||_1^2 e^{-\beta(f(B)+\lambda ||B||_1)}    }{    \sum_{B \in \mathcal{B}} ||B||_1 e^{-\beta(f(B)+\lambda ||B||_1)}    }$$
Noting that $\mathbb{E}||B||_1=:g(\lambda)$ and dividing the numerator and denominator of R.H.S. by $Z$, 
the condition is reduced to $\mathbb{E}||B||_1 \leq \frac{\mathbb{E}||B||_1^2}{\mathbb{E}||B||_1}$, which is true since 
$\mathbb{E}||B||_1^2 \geq (\mathbb{E}||B||_1)^2$. Hence, $\mathbb{E}||B||_1$ is decreasing in $\lambda$ for any $\beta>0$. Also, it is easy to verify that $|g'(\lambda)| \leq (\beta+1) N^2$. Hence, $g(\lambda)$ is Lipschitz continuous in $\lambda$.

\section{Proof of Theorem~\ref{theorem:optimality-of-the-learning-algorithm-for-constrained-problem}}
\label{appendix:proof-of-optimality-of-the-learning-algorithm-for-constrained-problem}
Let the distribution of $B(t)$ under GL  be 
$\pi^{(t)}(\cdot)$. Since $\lim_{t \rightarrow \infty} a(t)=0$, it follows that $\lim_{t \rightarrow \infty} d_V(\pi^{(t)}, \pi_{\beta| \lambda(t)})=0$ (where $d_V(\cdot,\cdot)$ is the total variation distance), and $\lim_{t \rightarrow \infty} ( \mathbb{E}_{\pi^{(t)}}||B(t)||_1-\mathbb{E}_{\pi_{\beta} | \lambda(t) }||B(t)||_1 ) :=\lim_{t \rightarrow \infty} e(t)=0$. Now, we can rewrite the $\lambda(t)$ update equation as follows:
\small
\begin{eqnarray}
 \lambda(t+1)=[  \lambda(t)+a(t) (\mathbb{E}_{\pi_{\beta}|\lambda(t)}||B(t)||_1-\bar{N}+M_t+e_t)  ]_b^c
\end{eqnarray} 
\normalsize

Here $M_t:=||B(t)||_1 - \mathbb{E}_{\pi^{(t)}} ||B(t)||_1$ is a Martingale difference noise sequence, and 
$\lim_{t \rightarrow \infty} e_t=0$. It is easy to see that the derivative of $\mathbb{E}_{\pi_{\beta} | \lambda }||B(t)||_1$ w.r.t. $\lambda$ is bouned for $\lambda \in (b,c)$; hence, $\mathbb{E}_{\pi_{\beta} | \lambda }||B(t)||_1$ is a Lipschitz continuous function of $\lambda$. It is also easy to see that the sequence $\{M_t\}_{t \geq 0}$ is bounded. Hence, by the theory presented in 
\cite[Chapter~$2$]{borkar08stochastic-approximation-book} and 
\cite[Chapter~$5$, Section~$5.4$]{borkar08stochastic-approximation-book}, $\lambda(t)$ converges to the unique zero of $\mathbb{E}_{\pi_{\beta} | \lambda }||B(t)||_1-\bar{N}$ almost surely. Hence,  $\lambda(t) \rightarrow \lambda^*$ almost surely. Since $\lim_{t \rightarrow \infty} d_V(\pi^{(t)}, \pi_{\beta| \lambda(t)})=0$ and $\pi_{\beta | \lambda}$ is continuous in $\lambda$, the limiting distribution of $B(t)$ becomes $\pi_{\beta | \lambda^*}$.

\section{Proof of Theorem~\ref{theorem:convergence-of-GLEM}}
\label{appendix:proof-of-GLEM}
The proof involves several steps, and  these steps are provided one by one. 
\subsubsection{Convergence in the fastest timescale}
Let us denote the probability distribution of $B(t)$ under GPL  by $\pi^{(t)}$ (a column vector indexed by the cofigurations from $\mathcal{B}$), and the corresponding transition probability matrix (TPM) by $A(t)$; i.e., $(\pi^{(t+1)})^T=(\pi^{(t)})^T A(t)=(1-1)\times(\pi^{(t)})^T + 1\times  (\pi^{(t)})^T A(t) $. This form is similar to a standard stochastic approximation scheme as in \cite[Chapter~$2$]{borkar08stochastic-approximation-book} except that the step size sequence for $\pi^{(t)}$ iteration is a constant sequence. Also, if $f^{(t)}(B)$, $\lambda(t)$ and $\theta(t)$ are constant with time $t$, then $A(t)=A$ will also be constant with time $t$, and the stationary distribution for the TPM $A$ will exist and will  be Lipschitz continuous in all (constant) slower timescale iterates. Hence, by using similar argument as in \cite[Chapter~$6$, Lemma~$1$]{borkar08stochastic-approximation-book}, one can show the following for all $B \in \mathcal{B}$:
\begin{eqnarray}\label{eqn:convergence-fastest-timescale}
\lim_{t \rightarrow \infty} |\pi^{(t)}(B)-\pi_{\beta,f^{(t)},\lambda(t), \theta(t)}(B)|=0 \,\,\, a.s.
\end{eqnarray}
where $\pi_{\beta,f^{(t)},\lambda(t), \theta(t)}(\cdot)$  can be obtained by replacing $h(B)$ in \eqref{eqn:definition-of-Gibbs-distribution} by $f^{(t)}(B)+\lambda(t)(\theta(t))||B||_1$

\subsubsection{Convergence of iteration \eqref{eqn:fB-update-iid-data}}
Note that,  \eqref{eqn:fB-update-iid-data} depends on $\theta(t)$ and not  on $B(t)$ and $\lambda(t)$; the iteration \eqref{eqn:fB-update-iid-data} depends on $\theta(t)$ through the estimation function $\mu_2$. Now, $f^{(t)}(B)$ is updated at a faster timescale compared to  $\theta(t)$. Let us consider the iterations \eqref{eqn:fB-update-iid-data} and \eqref{eqn:theta-update-iid-data}; they constitute a two-timescale stochastic approximation. 

Note that, for a given $\theta$, the iteration \eqref{eqn:fB-update-iid-data}  remains bounded inside a compact set independent of $\theta$; hence, using \cite[Chapter~$2$, Theorem~$2$]{borkar08stochastic-approximation-book} with additional modification as suggested in \cite[Chapter~$5$, Section~$5.4$]{borkar08stochastic-approximation-book} for projected stochastic  approximation, we can claim that $\lim_{t \rightarrow \infty} f^{(t)}(B) \rightarrow f_{\theta}(B)$ almost surely for all $B \in \mathcal{B}$, if $\theta(t)$ is kept fixed at a value $\theta$. Also, since $\mu_2 $ is Lipschitz continuous in $\theta$, we can claim that $f_{\theta}(B)$ is Lipschitz continuous in $\theta$ for all $B \in \mathcal{B}$. We also have $\lim_{t \rightarrow \infty}\frac{c(\nu(t))}{a(\nu(t))}=0$.

 Hence, by using an analysis similar to that in 
\cite[Appendix~E, Section~C.2]{chattopadhyay-etal15measurement-based-impromptu-deployment-arxiv-v1} (which uses \cite[Chapter~$6$, Lemma~$1$]{borkar08stochastic-approximation-book}), one can claim that:
\begin{eqnarray}\label{eqn:convergence_fB-iid-data}
\lim_{t \rightarrow \infty} |f^{(t)}(B)-f_{\theta(t)}(B)|=0 \,\,\, a.s. \,\,\, \forall B \in \mathcal{B}
\end{eqnarray}
{\em This proves the desired convergence of the iteration \eqref{eqn:fB-update-iid-data}.}

\subsubsection{Convergence of $\lambda(t)$ iteration}
 The $\lambda(t)$ iteration will view $\theta(t)$ as quasi-static and $B(t)$, $f^{(t)}(\cdot)$ iterations as equilibriated.

Let us  assume that $\theta(t)$ is kept fixed at $\theta$. Then, by \eqref{eqn:convergence-fastest-timescale} and \eqref{eqn:convergence_fB-iid-data}, we can work with $\pi_{\beta,f_{\theta},\lambda^{(t)}, \theta}$ in this timescale. 
Under this situation, \eqref{eqn:lambda-update-iid-data} asymptotically tracks the iteration  
$\lambda(t+1)=[\lambda(t)+b(t) ( \sum_{B \in \mathcal{B}} \pi_{\beta,f_{\theta},\lambda(t), \theta}(B) ||B||_1-\bar{N}+M_t )]_0^A$ where $\{M_t\}_{t \geq 0}$ is a Martingale differenece sequence. Now, $\pi_{\beta,f_{\theta},\lambda(t), \theta}(B)$ is Lipschitz continuous in $\theta$ and $\lambda(t)$ (using  Assumption~\ref{assumption:Lipschitz-continuity-wrt-theta}, Assumption~\ref{assumption:existence-of-lambda*} and a little algebra on the expression~\eqref{eqn:definition-of-Gibbs-distribution}). If $A_0$ is large enough, then, by the theory of \cite[Chapter~$2$, Theorem~$2$]{borkar08stochastic-approximation-book} and \cite[Chapter~$5$, Section~$5.4$]{borkar08stochastic-approximation-book}, one can claim that $\lambda(t) \rightarrow \lambda^*(\theta)$ almost surely, and $\lambda^*(\theta)$ is Lipschitz continuous in $\theta$ (by Assumption~\ref{assumption:existence-of-lambda*}).

 Hence, by using similar analysis as in 
\cite[Appendix~E, Section~C.2]{chattopadhyay-etal15measurement-based-impromptu-deployment-arxiv-v1} (which uses  \cite[Chapter~$6$, Lemma~$1$]{borkar08stochastic-approximation-book}), we can say that,  under iteration~\eqref{eqn:lambda-update-iid-data}:
\begin{eqnarray}\label{eqn:convergence-in-lambda-iteration-iid-data}
\lim_{t \rightarrow \infty} |\lambda(t)-\lambda^*(\theta(t))|=0 \,\,\, a.s.
\end{eqnarray}

\subsubsection{Convergence of the $\theta(t)$ iteration}
Note that, \eqref{eqn:theta-update-iid-data} is the slowest timescale iteration and hence it will view all other there iterations (at three different timescales) as equilibriated. However, this iteration is not affected by other iterations. Hence, this iteration is an example of simultaneous perturbation stochastic approximation as in \cite{spall92original-SPSA}, but with a projection operation applied on the iterates.  Hence, by combining \cite[Proposition~$1$]{spall92original-SPSA} and the discussion in \cite[Chapter~$5$, Section~$5.4$]{borkar08stochastic-approximation-book}, we can say that $\lim_{t \rightarrow \infty} \theta(t)=\theta^*$ almost surely in case Assumption~\ref{assumption:finite-number-of-local-maximum} holds.

If there does not exist a globally asymptotically stable equilibrium $\theta^*$ (as assumed in Assumption~\ref{assumption:finite-number-of-local-maximum}), then, by using the same techniques as in \cite[Appendix~E, Section~C]{chattopadhyay-etal15measurement-based-impromptu-deployment-arxiv-v1}, one can claim that the $\theta(t)$ iteration almost surely converges to the stationary points of the ODE $\dot{\theta}(\tau)=\bar
{\Gamma}_{\theta(\tau)}(\nabla g(\theta(\tau)))$.

\subsubsection{Completing the proof}
We have seen that  $\lim_{t \rightarrow \infty} \theta(t)=\theta^*$ almost surely. Hence, by \eqref{eqn:convergence-in-lambda-iteration-iid-data}, $\lim_{t \rightarrow \infty} \lambda(t)=\lambda^*(\theta^*)$ almost surely. By \eqref{eqn:convergence_fB-iid-data}, $\lim_{t \rightarrow \infty} f^{(t)}(B)=f_{\theta^*}(B)$ almost surely for all $B \in \mathcal{B}$. Then, by \eqref{eqn:convergence-fastest-timescale}, $\lim_{t \rightarrow \infty} \pi^{(t)}(B)=\pi_{\beta,f_{\theta^*},\lambda^*(\theta^*), \theta^*}(B)$ almost surely. Hence, Theorem~\ref{theorem:convergence-of-GLEM} is proved.

}

\end{document}